\documentclass{article}
\usepackage[lang=british]{ems-jfg}

\usepackage{etex,bm, amscd,amsfonts,amsthm, mathrsfs, tikz, multicol, subcaption}
\usepackage[justification=centering]{caption}
\usepackage[all,cmtip]{xy}

\newtheorem{theorem}{Theorem}[section]

\newtheorem{lemma}[theorem]{Lemma}

\newtheorem{remark}[equation]{Remark}

\theoremstyle{definition}
\newtheorem{definition}[equation]{Definition}
\newtheorem{example}[equation]{Example}

\newcommand{\E}{ {\mbox{$\mathcal{E}$}}}

\newcommand{\Z}{ {\mbox{$\mathbb{Z} $}}}

\newcommand{\R}{ {\mbox{$\mathbb{R} $}}} 
\newcommand{\C}{ {\mbox{$\mathbb{C} $}}} 
    
\newcommand{\He}{ {\mbox{$\mathbb{H} $}}}

\newcommand{\D}{ {\mbox{$\mathcal{D}$}}}

\newtheorem{prop}[theorem]{Proposition}

\newcommand{\Mink}{ {\mbox{$\overline{\textup{dim}}_B$}}}
\newcommand{\M}{ {\mbox{$\mathcal{M}$}}}  
\newcommand{\de}{ \delta} 
\newcommand{\ga}{ \gamma}

\newcommand{\om}{ \Omega }
\newcommand{\ze}{ \zeta} 
\newcommand{\zet}{ \widetilde{\zeta} }
\newcommand{\zeom}{ \zeta_A(\cdot,\om)} 
\newcommand{\ep}{\varepsilon} 

\numberwithin{equation}{section}

\begin{document}

\title{Fractal Zeta Functions and Complex Dimensions of Ahlfors Metric Measure Spaces}
\titlemark{Complex Dimensions of Ahlfors Spaces}



\emsauthor{1}{Michel L. Lapidus}{M.L. Lapidus}
\emsauthor{2}{Sean Watson}{S. Watson}


\emsaffil{1}{University of California Riverside, Department of Mathematics, Skye Hall,
	900 University Avenue, Riverside, CA 92521-0135, USA \email{lapidus@ucr.edu}}
\emsaffil{2}{California Polytechnic State University, Department of Mathematics, 1 Grand Ave, San Luis Obispo, CA 93407 \email{swatso01@calpoly.edu}}

\classification[46E30]{11M41, 28A75, 28A80}


\begin{abstract}
While classical analysis dealt primarily with smooth spaces, much research has been done in the last half century on expanding the theory to the nonsmooth case. Metric Measure spaces are the natural setting for such analysis, and it is thus important to understand the geometry of subsets of these spaces. In this paper we will focus on the geometry of Ahlfors regular spaces, Metric Measure spaces with an additional regularity condition. Historically, fractals have been studied using different ideas of dimension which have all proven to be unsatisfactory to some degree. The theory of complex dimensions, developed by M.L. Lapidus and a number of collaborators, was developed in part to better understand fractality in the Euclidean case and seeks to overcome these problems. Of particular interest is the recent theory of complex dimensions in higher-dimensional Euclidean spaces, as studied by M.L.Lapidus, G. Radunovi\'c, and D. \u Zubrini\'c, who introduced and studied the properties of the distance and tube zeta functions, $\zeta_A$ and $\zet_A$. We will show that this theory of complex dimensions naturally generalizes to the case of Ahlfors regular spaces, as the distance and tube zeta functions can be modified to apply to these spaces and all of its main properties carry over.  We also provide a selection of examples in Ahlfors spaces, as well as hints that the theory can be expanded to a more general setting. 
\end{abstract}

\maketitle


\section{Background}

While analysis has classically studied smooth spaces, much work has been done in the past half century in expanding the theory to the non-smooth cases. In particular, in 1971 Coiffman and Weiss introduced in \cite{CoWe} the notion of {\it spaces of homogeneous type}, a space with a quasimetric and a measure that satisfies a doubling condition, as the natural spaces for which the theory of Calder\'on--Zygmund singular integrals naturally extends. This suddenly presented a method to perform harmonic analysis  on these much more general sets. If we make the further restriction that the quasimetric be instead a metric, in which case the space will be called a {\it Metric Measure space} (or MM space for short). These spaces have since given rise to activity in many other analytical areas of mathematics, such as the study of function spaces, partial differential equations, probability theory, and analysis on fractals. See e.g.  \cite{DaMcCS},\cite{Hei}. 

\begin{definition}
	A {\it Metric Measure Space} (or MM space) is a set $X$ equipped with a metric $d$ and a positive Borel measure $\mu$ that is doubling; there exists a positive constant $C$ such that
	
	$$\mu(B_d(x,2r))\leq C\mu(B_d(x,r)),$$
	where $B_d(x,r)=B(x,r)$ denotes the closed ball of $X$ with center $x\in X$ and radius $r>0$.
\end{definition}

Ideally, we wish to understand the geometry of subsets of such spaces, particularly through the methods of fractal geometry. 

Historically, fractals have been studied under an approach to measure their ``size" and complexity (or ``roughness"), either through the Minkowski or Hausdorff dimensions. However, these dimensions leave many sets we would like to call fractal indistinguishable from standard Euclidean space. The theory of {\it complex dimensions}, developed by Lapidus and a number of collaborators since 1990, seeks to alleviate this problem and establish further geometric information intrinsic to a fractal; see, for example, \cite{Lap1,LaRa1,LaRaZu,LaRaZu2,LaRaZu3,LaRaZu4,LaRaZu5,LaRaZu6,LaRaZu7,LaRaZu8,LaRaZu9,LaSa,LaFra}  . In particular, both the real and imaginary parts of the complex dimensions encode information about the geometric, spectral and dynamical oscillations of a fractal space and its associated fractal drum (bounded open subset with fractal boundary).

In particular, we wish to generalize the recent theory introduced by Lapidus, Radunovi\'c, and \u Zubrini\'c in \cite{LaRaZu} (see also \cite{LaRaZu2}-\cite{LaRaZu9}). Here, the distance zeta function (and related zeta functions) are introduced, wherein bounded subsets $A$ of $\R^N$ are analyzed by studying the tubular neighborhood around the set. Given a bounded set $A \subseteq \mathbb{R}^n$, we define the {\em closed t-tubular neighborhood of }$A$ to be
\[ A_\delta = \{ x\in\R^N : d(x,A)\leq\delta \}. \] Note that the {\em open} tubular neighborhood was used in \cite{LaRaZu}, but the results hold in Euclidean space for closed tubular neighborhoods as well and this small modification will be required for the more general setting.

The {\em distance zeta function} is then defined, for all $s\in\mathbb{C}$ with $\text{Re}s$ sufficiently large,  as 
\[ \ze_A(s) = \int_{A_\de} d(x,A)^{s-N} \text{d}x .\]
This zeta function has many interesting properties that make it worthwhile to study. It is holomorphic in the open half-plane to the right of the upper Minkowski dimension, $D$ of set $A$ (in fact, $D$ is the abcissa of convergence of $\ze_A$), and if the function can be meromorphically continued to a connected open subset of $\C$ containing the right half-plane $\{\text{Re}s\geq D\}$, the poles of the zeta function give geometric information about $A$. We call the collection of poles the {\em complex dimensions} of $A$. 

In order to define such an analogous theory in more general settings, we first restrict ourselves to subsets of MM spaces that satisfy an extra regularity condition; they are called {\it Ahlfors regular spaces}. This added regularity condition roughly translates to requiring the measure of a ball to change consistently with any changes to the radius, regardless of where it is centered (see definition below). The way the measure changes as a power of the radius allows us to define an ambient dimension of the space in which we are working, called the {\it Ahlfors regularity dimension}.  This dimension allows us to justify the geometric information obtained by any fractal zeta function. In Chapter 5, we will study weakening this condition, and the problems that occur in doing so.

Here and thereafter, $B_d(x,r)$ denotes the closed ball of $(X,d)$ of center $x$ and radius $r$; often, we will omit the subscript $d$, when no confusion can follow.

\begin{definition}
	A MM space $(X,d,\mu)$ is {\it Ahlfors regular of dimension Q} (here on, {\it regular}) if $(X,d)$ is locally compact and $\mu$ locally finite with dense support satisfying, for some $K\geq 1$,
	
	$$K^{-1}r^Q\leq\mu(B(x,r))\leq Kr^Q ,$$ for all $x\in X, 0<r\leq\text{diam}X$.
	
	If only the upper (resp. lower) bound is satisfied, we call the space {\it upper (resp. lower) Ahlfors regular of dimension Q}.
\end{definition}

Despite requiring this additional regularity condition, many spaces that arise naturally in various areas of mathematics are Ahlfors regular. Important examples are the symbolic Cantor set, Laakso space and finite-dimensional Riemannian and sub-Riemannian manifolds; in particular, the Heisenberg group and Heisenberg type groups. Many self-similar fractals, such as the Sierpi\'nski gasket, are Ahlfors regular for the Hausdorff metric induced by the embedding in Euclidean space. However, in the theory of analysis on fractals, there are identifications with other spaces, such as the {\it harmonic Sierpi\'nski gasket}, which is Ahlfors regular for the Hausdorff measure associated with a natural geodesic metric (see \cite{Kaj}, \cite{Ki}).

The Hausdorff measure plays an important role in the study of Ahlfors spaces. This measure is defined on any metric space $(X,d)$ in the following manner:

\begin{definition}
	For any $s$ nonnegative, $A\subset X$, define the {\it $s$-dimensional Hausdorff outer measure}:
	\vspace{-.5em}
	$$H^s(A)=\lim_{\delta\rightarrow 0^+}H^s_\delta=\lim_{\delta\rightarrow 0^+}\inf\bigg\{\sum_{j=1}^\infty (\text {diam}E_j)^s : A\subseteq\bigcup_{j=1}^\infty E_j,\text {diam}E_j<\delta\bigg\}.$$
	The {\it Hausdorff dimension} $D_H$ of a set $A$ is defined by
	\begin{align*}
		D_H &=\inf\{s\geq 0:H^s(A)=0\} = \inf\{s\geq 0:H^s(A)<\infty\}	
		\\ &=\sup\{s\geq 0:H^s(A)=+\infty\}.
	\end{align*}
	We call $H^D$ the {\it D-dimensional Hausdorff measure} when restricted to Borel sets, where $D:=D_H$. Recall that the Hausdorff measure is a well defined $\sigma$-additive measure on the Borel $\sigma$-algebra of $X$; see e.g. \cite{Fol} or \cite[Theorem 6.14, p.207]{LaRa1}.
\end{definition}

It turns out that the measure $\mu$ of a regular $Q$-dimensional space and $H^Q$ are equivalent in the sense that there is a positive constant $C$ depending only on $K$ such that
$$C^{-1}\mu(E)\leq H^Q(E)\leq C\mu(E),~~~\text{ for all Borel sets } E\subseteq X.$$
In particular, if the MM space triple \((X,d,\mu)\) is regular of dimension $Q$, then so is \((X,d,H^D)\) and $Q=D_H$, the Hausdorff dimension of $(X,d)$..

Thus, we can always work under the assumption that we are using the $Q$-dimensional Hausdorff measure, and for ease of notation we will henceforth denote the measure of a set $A$ by $|A|$.

In Ahlfors spaces, we will also find that we need to subtly change the definition of the tubular neighborhood of a set:

\begin{definition}
	Let $\E$ be an Ahlfors regular space. Given $A\subseteq\E$ bounded, define the {\it $t$-neighborhood of $A$} (or {\it $t$-tubular neighborhood of $A$}) by $A_t:=\{x\in\E:d(x,A)\leq t\}$, where $t>0$. Note that in the setting of Ahlfors regular spaces, we take the {\em closed neighborhood}.
\end{definition}

Finally, we will need a proper definition of Minkowski content and Minkowksi dimension. Given a $Q$-regular Ahlfors space, we propose using the ambient dimension to generalize the Euclidean definition as follows:

\begin{definition}
	The {\it r-dimensional Minkowski upper content of} $A$ is defined as 
	\[\mathcal{M}^{*r}=\limsup_{t\rightarrow 0^+}\dfrac{|A_t|}{t^{Q-r}}.\]
	We then define the {\it upper Minkowski dimension} by
	\begin{align*}
		\overline{\text{dim}}_B A&=\inf\{r\geq 0:\mathcal{M}^{*r}(A)=0\}=\inf\{r\geq 0:\mathcal{M}^{*r}(A)<\infty\}
		\\ &=\sup\{r\geq 0:\mathcal{M}^{*r}(A)=+\infty\}.
	\end{align*}
	Similarly, we can define the {\it r-dimensional Minkowski lower content} as
	\[ \M_*^r = \liminf\limsup_{t\rightarrow 0^+}\dfrac{|A_t|}{t^{Q-r}} \]
	and the {\it lower Minkowski dimension} by
	\begin{align*}
		\underline{\textnormal{dim}}_B A &= \inf\{r\geq 0:\mathcal{M}_*^{r}(A)=0\} = \inf\{r\geq 0:\mathcal{M}_*^{r}(A)<\infty\}\\
		  &= \sup\{r\geq 0:\mathcal{M}_*^{r}(A)=+\infty\}.
	\end{align*}

	If $\overline{\text{dim}}_B A=\underline{\textnormal{dim}}_B A$, we say that $D=\text{dim}_B A:= \Mink A$ is the {\it Minkowski dimension}. If, further, $\M^{*D}=\M_*^D$, we call this common value the {\it D-dimensional Minkowski content of} $A$ and simply denote it by $\M^D$. If, in addition, $A$ is {\em Minkowski nondegenerate}, i.e., if $0< \M^D(A)<\infty$, then $A$ is said to be \emph{Minkowski measurable}.   
	
\end{definition}

\section{Fractal Zeta Functions}

In this section, we will study the basic properties of the distance zeta function and of the tube zeta function in Ahlfors spaces. In particular, our main result is that in Ahlfors spaces, the distance zeta function is holomorphic in the right half-plane $\{ \text{Re }s>\Mink A\}$, along with mild assumptions under which this bound cannot be improved. We also show that the abscissa of convergence of our distance zeta function on the right-hand side is $\Mink A$, the upper Minkowski dimension of the set $A$. We then obtain similar results for the tube zeta function.

The proof will require the following lemma originally stated, without proof, by Harvey and Polking in \cite{HarPol}. The following proof is due to \u Zubrini\'c in \cite{Zu3}, which uses dyadic decomposition. It is given in detail and extended to Ahlfors spaces for the needs of the theory in \cite{LaRaZu}. While originally written for Euclidean spaces, the proof still holds under the appropriate modifications for Ahlfors spaces.

\begin{lemma}[\textbf{Harvey--Polking Lemma}]\label{lem: 1}
	
	Assume that $A$ is an arbitrary bounded subset of an Ahlfors regular space $\E$ and let $\delta$ be any fixed positive number. If $\gamma$ is real and $\gamma\in(-\infty,Q-\Mink A)$, then
	\[\int_{A_\delta}d(x,A)^{-\gamma} \text{d}x<\infty .\] 
	
\end{lemma}

\begin{proof}
	
	If $\gamma\in (-\infty,0]$, then $x\mapsto d(x,A)^{-\gamma}$ is continuous and bounded on $A_{\delta}$ and the claim follows immediately. We will therefore focus on the case where $\Mink A<Q$. 
	
	Let $s\in(\Mink A, Q-\gamma)$ be arbitrary. Note that since $\gamma<Q-\Mink A$, this interval is nonempty. For $t\in(0,\delta]$, the function $t\mapsto |A_t|/t^{Q-s}$ is continuous, by the continuity of measures. Since $\M^{*s}(A)=0$ by assumption, the supremum of this function must be finite. Denoting the supremum by $C(\de)$,  then $|A_t|\leq C(\de)t^{Q-s}$, for all $t\in(0,\de)$.
	
	We will now use a type of dyadic decomposition of the set $A_{\de}\setminus\overline{A}$:
	\[A_{\de}=\overline{A}\cup\Big(\bigcup_{i=1}^\infty B_i\Big),~~ B_i:=A_{2^{-i}\de}\setminus A_{2^{-i-1}\de} .\]
	
	Since we have assumed $\Mink A<Q$, we have that $|\overline{A}|=0$. Thus
	\[ I(A):=\int_{\overline{A}} d(x,A)^{-\ga} \text{d}x=0 .\]
	
	If $\de>1$, then $d(x,A)^{-\ga}$ is bounded on $A_\de\setminus A_1$. Thus, we can assume $\de\in(0,1]$. Using our dyadic decomposition and the assumption $0<\ga<Q-s$, we have:
	
	\begin{align*}
		\int_{A_{\delta}} d(x,A)^{-\ga} \text{d}x  &= I(A)+\sum_{i=0}^\infty \int_{B_i}d(x,A)^{-\ga}\text{d}x \\
		&\leq I(A)+\sum_{i=0}^\infty\int_{A_{2^{-i}\delta}}d(x,A)^{-\ga}\text{d}x \\
		&\leq I(A)+\sum_{i=0}^\infty (2^{-i-1}\de)^{-\ga}|A_{2^{-i}\de}| \\
		&\leq I(A)+C(\de)\sum_{i=0}^\infty (2^{-i-1}\de)^{-\ga}(2^{-i}\de)^{Q-s} \\
		&\leq I(A)+ \frac{2^\ga C(\de)}{1-2^{\gamma-Q+s}}\de^{Q-s-\gamma}<\infty.
	\end{align*}
	
	This completes the proof.
\end{proof}

We can think of this lemma as an extension of the fact that, in $\R^N$, if $\gamma<N$ then the integral of $|x|^{-\gamma}$ over the unit ball is finite. 

The left-hand integral of the Harvey--Polking lemma can be viewed as a prototype of the distance zeta function, where the exponent is only real-valued. In order to prove a useful identity, we will need the following theorem, stated in Folland's classic real analysis text \cite{Fol}.

\begin{lemma}[\text{G.B. Folland, \cite[p. 198]{Fol}}]
	Let $f\in(X,\M,\mu)$ be a nonnegative measurable function in a measure space $X$, i.e., $f:X\rightarrow [0,+\infty]$, and let $0<\alpha<\infty$.   Then
	\[\int_X f(x)^\alpha\text{d}x = \alpha\int_0^{+\infty} t^{\alpha-1}|\{f>t\}|\text{d}t,\]
	where $\{f>t\}:=\{x\in X: f(x)>t\}$ and $|\{f>t\}|$ denotes the $M$-dimensional Lebesgue measure of $\{f>t\}$.
\end{lemma}

This next lemma, which is an exact counterpart of \cite[Lemma 2.1.4]{LaRaZu}, establishes the identity that changes the integral over the distance function to the integral of the tubular volume of $A_t$. We can view this, again, as a prototype of the {\em tubular zeta function} defined further on in Definition \ref{tube}.

\begin{lemma} \label{th: 2}
	Let $A$ be any bounded set in an Ahlfors regular space $\E$. Then, for every value of the real parameter $\ga$ in the open interval $(-\infty, Q-\Mink A)$, the following identity holds:
	\[\int_{A_\delta}d(x,A)^{-\ga}\text{d}x=\de^{-\ga}|A_\delta|+\ga\int_0^\delta t^{-\ga-1}|A_t|\text{d}t.\]
	Furthermore, both integrals in the above identity are finite.
\end{lemma}

\begin{proof}
	We consider the following three cases: \smallskip
	
	$(a)$ Case where $\ga>0$: Since $0<\ga<Q-\Mink A$, we will proceed similarly to \cite{Zu5}.  We use Lemma 2.2 with $\alpha=\ga$ and the Borel measurable function $f:\E\rightarrow[0,+\infty]$ given by
	\[f(x):= \begin{cases} d(x,A)^{-1}, & \mbox{for } x\in A_\delta, \\
		0, & \mbox{for } x\in\E\setminus A_\delta . \end{cases}  \]
	By definition, $f(x)=+\infty$ for $x\in\overline{A}$; further, since $\Mink A<Q$, we have $|\overline{A}|=0$. Note that the set $\{x\in\E:f(x)>t\}$ is equal to $A_{1/t}$ for $t>\de^{-1}$ and to the constant set $A_\delta$ for $t\in(0,\de^{-1})$. Therefore,
	\begin{align*}
		\int_{A_\delta}d(x,A)^{-\ga}\text{d}x &= \ga\Big(\int_0^{1/\delta}+\int_{1/\delta}^{+\infty}\Big) t^{\ga-1}|\{f>t\}|\text{d}t \\
		&= \ga|A_\delta|\int_0^{1/\delta}t^{\ga-1}+\ga\int_{1/\delta}^{+\infty}t^{\ga-1}|A_{1/t}|\text{d}t .
	\end{align*}
	The result follows by a change of variable $\tau=1/t$ in the last integral. In order to show this integral is finite, let $\ep>0$ be small enough so that $\ga\in(0,Q-D-\ep)$, where $D:=\Mink A$. Then $\M^{*(D+\ep)}=0$, and so there exists a positive constant $C=C(\delta)$ such that $|A_t|\leq Ct^{Q-d-\ep}$ for all $t\in(0,\de]$. Hence,
	\[ \int_0^\de t^{-\ga-1}|A_t|\text{d}t\leq C\int_0^\de t^{Q-D-\ep-\ga-1}\text{d}t<\infty .\] \smallskip
	
	$(b)$ Case where $\ga=0$: If we assume that $\ga=0<Q-\Mink A$, then it suffices to show that $I:=\int_0^\de t^{-1}|A_t|\text{d}t<\infty$. Letting $D:=\Mink A$, we then have $D+\ep<Q$ for $\ep>0$ small enough; hence, since $\M^{*(D+\ep)}(A)=0$, there exists a positive constant $C$ such that $|A_t|\leq Ct^{Q-D-\ep}$ for all $t\in(0,\de)$. This immediately implies that $I\leq C\int_0^\de t^{Q-D-\ep-1}\text{d} t<\infty$. \smallskip
	
	$(c)$ Case where $\ga<0$: In this case, the left-hand side of our equation is clearly finite. We shall use Lemma 2.2 with $\alpha=-\ga$ and the Borel measurable function $f:\E\rightarrow[0,+\infty]$ given by 
	\[f(x):= \begin{cases} d(x,A), & \mbox{for } x\in A_\delta, \\
		0, & \mbox{for } x\in\E\setminus A_\delta . \end{cases}  \]
	Note that $\{f>t\}=\emptyset$ for $t\geq\de$, and $\{f>t\}=A_\de\setminus A_t$ for $0<t<\de$. Thus for $t<\de$, we have $|\{f>t\}|=|A_\de|-|A_t|$. We obtain
	\[\int_{A_\de} d(x,A)^\alpha \text{d}x=\alpha\int_0^\de t^{\alpha-1}(|A_\de|-|A_t|)\text{d}t=\de^\alpha|A_\de|-\alpha\int_0^\de t^{\alpha-1}|A_t|\text{d}t,\]
	which, after replacing $\alpha$ by $-\ga$, completes the proof.
	
\end{proof}

The proof of Lemma \ref{th: 2} just above bears a small but important difference to that in the Euclidean case. In particular, part (c) requires (for now) that the tubular neighborhood $A_t$ be closed in the Ahlfors regular setting, which is not the case in $\R^N$. This is due to a result in \cite{Sta} which states that $|A_t|=|\overline{A_t}|$ in the Euclidean setting, but which we conjecture to be false in the more general Ahlfors setting.

The following lemma (analog of \cite[Lemma 2.1.6]{LaRaZu}) complements the Harvey--Polking theorem, establishing the domain of integrability in the real setting: 

\begin{lemma}\label{lem: 4}
	Let $A$ be a bounded subset of $\E$ and $\de>0$. If $\ga>Q-\Mink A$, then $\int_{A_\de} d(x,A)^{-\ga}\text{d}x=+\infty$.
\end{lemma}

\begin{proof}
	For all $\ga>0$, 
	\[ I_\de:=\int_{A_\de} d(x,A)^{-\ga}\text{d}x=\de^{-\ga}|A_\de|+\ga\int_0^\de s^{-\ga-1}|A_s|\text{d}s\geq \de^{-\ga}|A_\de|,\]
	where the second equality holds from the same proof as in Lemma 2.3, case $(a)$.
	Let $D:=\Mink A$ and choose a positive $\sigma<D$ sufficiently close to $D$ so that $\ga>Q-\sigma$. Then, for all $\sigma$ in $(0,\delta)$, $\M^{*\sigma}(A)=+\infty$, which implies that there exists a sequence of positive numbers $s_k$ converging to zero such that
	\[ C_k:=\frac{|A_{s_k}|}{s_k^{Q-\sigma}}\rightarrow+\infty ~~\text{as}~~k\rightarrow+\infty .\]
	Since the map $\de\mapsto I_\de$ is nondecreasing on $(0,+\infty)$, we have for all $k$ large enough
	\[ I_\de\geq I_{s_k}\geq (s_k)^{-\ga}|A_{s_k}|=C_k\cdot s_k^{Q-\sigma-\ga}\rightarrow+\infty\]
	as $k\rightarrow\infty$. Hence, $I_\de=+\infty$.
\end{proof}

Given these lemmas, which deal with the soon defined distance zeta function restricted to real numbers, we can state our main theorem, which is an exact counterpart in our more general context of \cite[Theorem 2.1.11]{LaRaZu}. The proof is included for completeness, but is nearly identical to the Euclidean case.

\begin{theorem}\label{th: 5}
	Let $A$ be an arbitrary bounded subset of $\E$ and let $\de>0$. Then:
	\smallskip
	
	$(a)$ The $\textup{distance zeta function of}$ $A$ (with parameter $\delta$), denoted $\ze_A$ and defined, for all $s\in\mathbb{C}$ with $\text{Re}s$ sufficiently large, by 
	\[\ze_A(s)=\int_{A_\de} d(x,A)^{s-Q}\textnormal{d}x, \]
	is holomorphic in the open half-plane $\{\textnormal{Re} s>\Mink A\}$, and for all complex numbers $s$ in that region,
	\[\ze_{A_\de}'(s)=\int_{A_\de} d(x,A)^{s-Q}\log d(x,A)\textnormal{d}x .\]
	
	$(b)$ The lower bound of absolute convergence is optimal, i.e.
	\[ \Mink A=D(\ze_A),\]
	where $D(\ze_A)$ is the abscissa of absolute convergence of $\ze_A$; i.e., $\{\text{Re}s>\overline{\text{dim}}_B A\}$ is the largest open half-plane on which the Lebesgue integral $\int_{A_\delta}d(x,A)^{s-Q} \text{d}x$ is convergent (and thus also, absolutely convergent).
	
	$(c)$ If the Minkowski dimension $D:=\text{dim}_B A$ exists, $D<Q$, and $\M_*^D(A)>0$, then $\ze_A(s)\rightarrow+\infty$ as $s\in\R$ converges to $D$ from the right. In this case, \textup{the abscissa of holomorphic continuation} of $\ze_A$, denoted $D_{\text{hol}}(\ze_A)$, coincides with the abscissa of absolute convergence of $\ze_A$; i.e.,$\{\text{Re}s>\overline{\text{dim}}_B A\}$ is the largest open half-plane to which $\ze_A$ can be holomorphically continued and so, $D_{\text{hol}}(\ze_A)=\overline{\text{dim}}_B A$.
\end{theorem}

\begin{proof}
	$(a)$ Let $I(s):=\int_{A_\de} d(x,A)^{s-Q}\log d(x,A)\text{d}x$. To prove the holomorphicity of $\ze_A$ and the corresponding statement about the complex derivative of $\ze_A$, let us fix any $s \in \mathbb{C}$ such that $\text{Re} s>\Mink A$. We thus have to show that 
	\begin{align*}
		R(h) &:= \frac{\ze_A(s+h)-\ze_A(s)}{h}-I(s)\\
		&= \int_{A_\de}\Big(\frac{d(x,A)^h-1}{h}-\log d(x,A)\Big)d(x,A)^{s-Q}\text{d}x
	\end{align*}
	converges to zero as $h\rightarrow 0$ in $\C$, with $h\neq 0$.
	
	Letting $d:=d(x,A)\in(0,\de)$, we first consider
	\[f(h):=\frac{d^h-1}{h}-\log d= \frac{1}{h}(e^{h\log d}-1)-\log d.\]
	Using the MacLaurin series $e^z=\sum_{j\geq 0}\frac{z^j}{j!}$, which converges for all $z\in\C$, we obtain
	\[f(h)=h(\log d)^2\sum_{k=0}^\infty\frac{1}{(k+2)(k+1)}\cdot \frac{(\log d)^k h^k}{k!},\]
	for all $h\in\C$. Further, assuming without loss of generality that $\de\in(0,1]$, and hence $\log d\leq 0$, we have
	\begin{align*}
		|f(h)| &\leq \frac{1}{2}|h|(\log d)^2\sum_{k=0}^\infty\frac{(|\log d| |h|)^k}{k!} \\
		&= \frac{1}{2}|h|(\log d)^2 e^{-|h|\log d} \\
		&= \frac{1}{2} |h|(\log d)^2 d^{-|h|} .
	\end{align*}
	Thus,
	\[|R(h)|\leq \frac{1}{2}|h|\int_{A_\de}|\log d(x,A)|^2 d(x,A)^{\text{Re} s-Q-|h|}\text{d}x.\]
	Let $\ep>0$ be sufficiently small, to be specified. Taking $h\in\C$ such that $|h|<\ep$, since $\de\leq 1$, and hence $d(x,A)\leq 1$, we have
	\[|R(h)|\leq \frac{1}{2}|h|\int_{A_\de}|\log d(x,A)|^2 d(x,A)^\ep d(x,A)^{\text{Re} s-Q-2\ep}\text{d}x.\]
	By a simple application of L'H\^opital's Rule, it is clear that there exists a positive constant $C=C(\de,\ep)$ such that $|\log \rho|^2 \rho^\ep\leq C$ for all $\rho\in(0,\de)$. This implies that
	\[|R(h)|\leq \frac{1}{2} C|h|\int_{A_\de} d(x,A)^{\text{Re} s-Q-2\ep}\text{d}x.\]
	Letting $\ga:=2\ep+Q-\text{Re} s$, we see that the integrability condition $\ga<Q-\Mink A$ is equivalent to $\text{Re} s>\Mink A+2\ep$. Since $s$ is fixed with $\text{Re} s>\Mink A$, we only need specify that $\ep$ be small enough to fulfill this inequality. Then $R(h)\rightarrow 0$ as $h\rightarrow 0$ in $\C$, with $h\neq 0$. Therefore we conclude that $\ze_A(s)$ is holomorphic for $\text{Re} s>\Mink A$, with derivative $\ze_A' (s)$ given as desired.
	
	\smallskip
	$(b)$ This follows immediately from part $(a)$ and Lemma \ref{lem: 4}.
	
	\smallskip
	
	$(c)$ Since $\M_*^D(A)>0$, then for any $\de>0$ there exists $C>0$ such that for all $t\in (0,\de)$, we have $|A_t|\geq C t^{Q-D}$. Using Lemmas 2.1 and 2.2, we see that for any $\ga\in (0,Q-D)$,
	\begin{align*}
		\infty &> I(\ga)=\int_{A_\de} d(x,A)^{-\ga}\text{d}x =\de^{-\ga}|A_\de|+\ga\int_0^\de t^{-\ga-1}|A_t|\text{d}t \\
		&\geq \ga C\int_0^\de t^{Q-D-\ga-1}\text{d}t = \ga C \frac{\de^{Q-D-\ga}}{Q-D-\ga} .
	\end{align*}
	Therefore, if $\ga\rightarrow Q-D$ from the left, then $I(\ga)\rightarrow +\infty$. Equivalently, if $s\in\R$ is such that $s\rightarrow D$ from the right, then $\ze_A(s)\rightarrow +\infty$.
	
\end{proof}

We point out that, alternatively, part $(a)$ of Theorem \ref{th: 5} could be established by applying the well-known theorem concerning integrals depending holomorphically on a complex parameter, much as in the second proof given in \cite{LaRaZu}.

As noted in the lemmas and  is clear from the definition, if $s$ is real then $\ze_A$ is also real-valued. Futhermore, using the principle of reflection, we have that for any complex number such that $\text{Re } s>\Mink A$,  $\overline{\ze_A(s)}=\ze_A(\overline{s})$. We will be mostly concerned with meromorphically extending the distance zeta function to suitable domains, when possible, and below we will define the {\em visible complex dimensions of the set A} as the poles in such domains of these extensions. Thus, if the distance zeta function can be meromorphically extended to a region symmetric with respect to the real axis, the nonreal (visible) complex dimensions come in conjugate pairs.

Given that our bounded subset $A$ is typically not of ``full dimension'' in the sense that $\Mink A<Q$, we have the following proposition (analog of \cite[Proposition 2.1.13]{LaRaZu}):

\begin{prop}
	Assuming that $|\overline A|=0$ (which is always the case if $\Mink A<Q$), and given any $\de >0$, we can compute the distance zeta function $\ze_A$ as follows for all $s\in\C$ with $\textnormal{Re} s>\Mink A$:
	
	\[ \ze_A(s)=\lim_{\ep\rightarrow 0^+} \int_{A_\de\backslash A_\ep} d(x,A)^{s-Q} \textnormal{d}x. \]
	
\end{prop}

\begin{proof}
	The characteristic function $\chi_{A_\de\setminus A_\ep}$ converges to $\chi_{A_\de}$ a.e. in $A_\de$ as $\ep\rightarrow 0^+$. To show that our limit converges uniformly, we take a complex number $s$ and real number $\xi$ such that $\text{Re} s > \xi >\Mink A$, and assume without loss of generality that $\ep\in (0,1)$, noting that 
	\[ \Bigg{|}\int_{A_\ep} d(x,A)^{s-Q} \text{d}x\Bigg{|} \leq \int_{A_\ep} d(x,A)^{\text{Re} s-Q} \text{d}x \leq \int_{A_\ep} d(x,A)^{\xi-Q} \text{d}x .\]
	Let us choose any $d\in (\Mink A, \xi)$. Since $d>\Mink A$, we have that $\M^{*d}(A)=0$, which means there exists a positive constant $C=C(d,Q,A)$ such that $|A_t|\leq Ct^{Q-d}$ for all $t\in (0,\ep].$ Using the (extended) Harvey-Polking lemma (Lemma \ref{lem: 1})with $\ga := Q-\xi$, it follows that
	\begin{align*}
		\int_{A_\ep} d(x,A)^{\xi-Q} \text{d}x &= \ep^{-\ga}|A_\ep|+\ga\int_0^\ep t^{\ga-1}|A_t|\text{d}t \\
		&\leq \ep^{\ga}C\ep^{Q-d}+\ga\int_0^\ep t^{\ga-1}Ct^{Q-d} \text{d} t = C_1\ep^{\xi-d},
	\end{align*}
	where $C_1 := C(n-d)/(\xi-d)$. Hence, using $d<\xi$, we conclude that
	\[ \sup_{\text{Re}s>\xi}\Big{|}\int_{A_\ep}d(x,A)^{s-Q}\text{d}x\Big{|}\leq C_1e^{\xi-d}\rightarrow 0^+~~~\text{as }\ep\rightarrow 0^+ ,\]
	and our result follows.
\end{proof}

From this proof, we can deduce that the distance zeta function satisfies the following asymptotic property (analog of \cite[Proposition 2.1.14]{LaRaZu}).

\begin{prop}
	Assume that $A$ is a bounded subset of $\E$, $\ep\in (0,1)$, and let us define the fixed-epsilon distance zeta function $\ze_A(\cdot,A_\ep)$ by
	\[ \ze_A(s,A_\ep):= \int_{A_\ep} d(x,A)^{s-Q}\text{d}x.\]
	Then, for any $\xi>\Mink A$ and $d\in(\Mink A,\xi),$ there exists a positive constant $C_1=C_1(\xi,d,Q,A)$ such that 
	\[ \sup_{\text{Re}s>\xi}|\ze_A(s,A_\ep)|\leq C_1\ep^{\xi-d},~~~\text{for all } \ep\in (0,1).\]
	In other words, $\sup_{\text{Re}s>\xi}|\ze_A(s,A+\ep)|=O(\ep^{\xi-d})$ as $\ep\rightarrow 0^+.$ 
\end{prop}

The following theorem, counterpart of \cite[Lemma 2.1.15]{LaRaZu}, establishes that the distance zeta function is an entire function away from the set $A$. Note that since the zeta function is associated with the ordered pair $(A,U)$, where $A$ and $U$ are suitable subsets of $\E$, we could establish this theorem in the manner of {\em general relative fractal drums}. These are studied further at the end of the section, as introduced in \cite{LaRaZu} and extended to the present setting.

\begin{theorem}
	\label{th: 8}
	Let $A$ and $U$ be bounded subsets of $\E$ which have disjoint closures, i.e., $\overline{A}\cap\overline{U}=\emptyset.$ Further assume that $U$ is Hausdorff measurable, and hence, that $U$ is a Borel subset of $\E$; see, e.g., \cite[Theorem 6.14, p. 207]{LaRa1}. Then
	\[ F:\C\rightarrow\C,~~~F(s):=\int_U d(x,A)^{s-Q}\text{d}x. \]
	is an entire function and we have
	\[ F'(s)=\int_U d(x,A)^{s-Q} \log d(x,A)\text{d}x, \]
	for all $s\in\C .$
\end{theorem}

\begin{proof}
	Let $s$ be a fixed complex number and set $R(h)=\frac{1}{h}(F(s+h)-F(s))-I_1(h)$, for $h\in\C,~h\neq\emptyset,$ where $I_1:=\int_U d(x,A)^{s-Q} \log d(x,A)\text{d}x.$ We follow the same procedure as in the proof of Theorem \ref{th: 5}, part (a), where it follows that
	\[ |R(h)|\leq \frac{1}{2} |h| \int_U |\log d(x,A)|^2 \exp(|\log d(x,A)| |h|) d(x,A)^{\text{Re} s-Q}\text{d}x. \]
	Since $A$ and $U$ are disjoint and bounded, there exists positive constants $d_1$ and $d_2$ such that $d_1\leq d(x,A)\leq d_2$ for all $x\in U$. Therefore, the function under the integral sign above is bounded from above by a constant $C$, uniformly for all $h\in\C$ such that  $|h|\leq\ep$, where $\ep >0$ is fixed:
	\begin{align*}
		C=&\max\{(\log d_1)^2,(\log d_2)^2\}\exp(\max\{|\log d_1|,|\log d_2|\}\ep) \\
		&\cdot \max\{d_1^{\text{Re}s-Q},d_2^{\text{Re}s-Q}\}.
	\end{align*}
	Hence, $|R(h)|\leq \frac{1}{2} |h| C |U|$, and therefore $R(h)\rightarrow 0$ as $h\rightarrow 0$ in $\C$, with $h\neq 0$.
\end{proof}

We also wish to generalize Lemma \ref{th: 2} for the distance zeta function in full. The established identity is used below to define the {\em tube zeta function of A}, which presents an alternate zeta function for analysis. This result is the counterpart to \cite[Theorem 2.2.1]{LaRaZu}. We will see that, in certain examples, the tube zeta function is much easier to use compared to the distance zeta function.

\begin{theorem}\label{funct}
	Let $A$ be any bounded subset of an Ahlfors regular space $\E$. Then, for every value of the parameter $s\in\C$ such that $\textnormal{Re } s>\Mink A$, the following identity holds:
	\begin{equation}\label{eq: 1}
		\int_{A_\delta}d(x,A)^{s-Q}\text{d}x=\de^{s-Q}|A_\delta|+(Q-s)\int_0^\delta t^{s-Q-1}|A_t|\text{d}t.
	\end{equation}
\end{theorem}

\begin{proof}
	Let $D:=\Mink A$. By Lemma \ref{th: 2}, we know that this identity already holds for all $s\in\R$ such that $s>D.$ Let us denote the left-hand side of Equation (\ref{eq: 1}) by $f(s)$ and the right-hand side by $g(s)$. Since $f(s)=g(s)$ on the subset $(D, +\infty)\subset\C$ (which has an accumulation point in $\C$), to prove the theorem, it suffices to show that $f(s)$ and $g(s)$ are both holomorphic in the region $\{\text{Re }s>D\}$; the result will then follow by the principle of analytic continuation. Note that the holomorphicity of $f(s)$ in this region was already proven in part $(a)$ of Theorem \ref{th: 5}.
	
	To prove the holomorphicitiy of $g(s)$ on $\{\text{Re }s>D\}$, it suffices to consider $g_1(s)=\int_0^\de t^{s-Q-1}|A_t|\text{d}t$. Note that $g_1(s)$ has the form of a Dirichlet-type integral, in the sense of \cite[Appendix A]{LaRaZu}: $g_1(s)=\int_E\phi(t)^s\text{d}\mu (t),$ where $E:=(0,\de)$, $\phi(t):=t,$ and the positive measure $\mu$ is given by $\text{d}\mu(t):=t^{-Q-1}|A_t|\text{d}t.$ Thus it suffices to show that on this region, $g_1(s)$ is well defined. To prove this, let $\ep>0$ be small enough so that $\text{Re }s>D+\ep$. Since $\M^{*(D+\ep)}(A)=0,$ there exists $C_\de>0$ such that $|A_t|\leq C_\de t^{Q-D-\ep}$ for all $t\in (0,\de].$ Then
	\begin{align*}
		|g_1(s)|&\leq\int_0^\de t^{\text{Re }s-Q-1}|A_t|\text{d}t \\
		&\leq C_\de\int_0^\de t^{\text{Re }s-D-\ep-1} \text{d}t=C_\de\dfrac{\de^{\text{Re }s-D-\ep}}{\text{Re }s-D-\ep}<\infty.
	\end{align*}
	For fixed $s_0\in\C$, with $\text{Re}s_0>D$ and for $\eta>0$ sufficiently small, it follows that $g_1(s)$ is Lebesgue integrable for $|\hspace{0.1em}\text{Re}(s-s_0)|<\eta$. One then deduces from the classic theorem about an integral depending holomorphically on a parameter that $g$ is holomorphic at $s_0$ and hence also, on the open half-plane $\{\text{Re}s>D\}$.
\end{proof}

\begin{definition}\label{tube}
	Let $\de$ be a fixed positive number, and let $A$ be a bounded subset of $\E$. Then the {\it tube zeta function of $A$}, denoted by $\zet_A$, is defined by
	\[ \zet_A(s)=\int_0^\de t^{s-Q-1}|A_t|\text{d}t,\]
	for $s\in\C$ with $\text{Re }s$ sufficiently large.
\end{definition}

We also denote the {\em abscissa of holomorphic continuation of the function f} as $D(f)$ for the sequel. Importantly, we know that if a bounded set $A$ is such that $D:=\text{dim}_B A$ exists, $D<Q$, and $\M_*^D(A)>0$, then $D(\ze_A)=\Mink A$; see Theorem \ref{th: 5}, part (c). The following lemma (analog of \cite[Lemma 2.1.52]{LaRaZu}) proves that $D(\ze_A) \geq 0$ for any bounded subset $A$ of $\E$, which is not the case for relative fractal drums. In particular, there are examples of relative fractal drums in $\R^N$ that have negative abscissa of convergence, see \cite[Subsection 4.1.2]{LaRaZu}.

\begin{lemma}
	For any bounded subset $A$ of $\E$, we have $D(\ze_A)\geq 0.$
\end{lemma}

\begin{proof}
	Assume to the contrary that $D(\ze_A)<0.$ Then $\ze_A(s)$ is well defined and continuous for $s\in(D(\ze_A),+\infty)$, and in particular, it is continuous at $s=0$.
	
	Let us take any $a\in A$. Since $A_\de\supset B_\de(a)$, and $d(x,A)\leq d(x,a),$ we have that for every real number $s\in(0,Q)$,
	\begin{align*}
		\zet_A(s)    &=\int_0^\de t^{s-Q-1}|A_t|\text{d} t\geq \int_0^\de t^{s-Q-1}|B_t (a)|\text{d}t\\
		&\geq\int_0^\de t^{s-Q-1} K^{-1}t^Q\text{d}t=K^{-1}\dfrac{\de^s}{s},
	\end{align*}
	where $K^{-1}t^Q$ comes from the bound on the measure of the ball $B_t (a)$ given by the Ahlfors regularity condition. Thus, $\zet_A(s)$ (and as a consequence of Theorem \ref{funct} above, $\ze_A(s)$) $\rightarrow +\infty$ as $s\rightarrow 0^+,~s\in\R$. This clearly contradicts the continuity of $\ze_A$ at $s=0$.
\end{proof}

We can now define the {\em complex dimensions of a set A}, provided that the distance or tube zeta functions can be meromorphically extended to an {\em admissible} domain:

\begin{definition}
	A bounded subset $A$ of $\E$ such that $\ze_A$ can be meromorphically extended to an open domain $G$ containing the closed half-plane $\{ \text{Re }s\geq D(\ze_A)\}$ is called {\it admissible}.
\end{definition}

\begin{definition}
	Given an admissible set $A$, we consider the {\it set of poles of $\ze_A$ located on the critical line} $\{\text{Re }s=D(\ze_A)\}$:
	\[ \mathscr{P}_c (\ze_A)=\{\omega\in W:\text{$\omega$ is a pole of of $\ze_A$ and Re $\omega=D(\ze_A)$}\}, \]
	called the {\it set of principal complex dimensions of $A$}.
	
	We call the set of all poles in the region $G$ of meromorphic extension of $\ze_A$ the {\it set of visible complex dimensions of $A$ (with respect to G)}, and we denote it by $\mathscr{P}(\ze_A)$ or $\mathscr{P}(\ze_A,G)$:
	\[ \mathscr{P}(A)=\{\omega\in G : \text{$\omega$ is a pole of $\ze_A$} \} \]
	
\end{definition}

It may not be convenient or even possible to meromorphically extend $\ze_A$ to all of $\C$. These complex dimensions give important geometric information about the set $A$. It is important that the complex dimensions are given as a {\em set} (or even as a multiset), as each dimension can be interpreted to represent a fuller picture into the geometry. For example, given a simple simplicial 3-complex composed of vertices, edges, and faces, the complex dimensions given by the distance zeta function would be $\{ 0,1,2 \}$. 

While complex dimensions can be useful for standard geometric subsets, subsets that exhibit {\em geometric oscillations}, such as Cantor sets, produce nonreal complex dimensions. As proposed by Lapidus and collaborators (see, e.g., \cite{LaFra}, \cite{LaRaZu}, \cite{Lap1}), a fractal set is a set that has nonreal complex dimensions (or more complex behavior, such as essential singularities along the vertical line $\{ \text{Re }s=D(\ze_A) \}$).

As stated above, depending on the set we wish to study, it may be preferential to use the tube zeta function. Before we can do so, we first establish their equivalence in terms of complex dimensions as defined in \cite{LaRaZu}:

\begin{definition}
	Two meromorphic functions $f$ and $g$ on a domain $G\subset\C$ are said to be {\it equivalent} if $D(f)=D(g)$, and their sets of poles contained on the common vertical line $\{ \text{Re }s=D(f)=D(g)\}$ coincide:
	\[ f\sim g \iff D(f)=D(g) ~\text{and } \mathscr{P}_c(f)=\mathscr{P}_c(g). \]
\end{definition}

Now, the relationship between the distance and tube zeta functions is given by
\[ 	\int_{A_\delta}d(x,A)^{s-Q}\text{d}x=\de^{s-Q}|A_\delta|+(Q-s)\int_0^\delta t^{s-Q-1}|A_t|\text{d}t.\]
As long as $\Mink A < Q$, the two zeta functions are equivalent. However, if $\Mink A = Q$, it is possible that the two zeta functions differ by a pole at $s=Q$. The tube zeta function can pick up such poles, but the distance zeta function misses these poles of full dimension. Thus the two functions are almost equivalent, as long as this caveat is kept in mind. 

Actually, much more is true. Namely, on a given domain $G$ of $\mathbb{C}$, $\zet_A$ has a (necessarily unique) meromorphic continuation to $G$ if and only if $\ze_A$ does, and in that case, $\ze_A$ and $\zet_A$ are connected via the fundamental equation (\ref{eq: 1}) of Theorem \ref{funct} on all of $G$: for all $s\in G$, the following {\em functional equation} holds: 
\begin{equation}\label{functeqn}
	\ze_A(s)=\delta^{s-Q}|A_\delta|+(Q-s)\zet_A(s) .
\end{equation} 
As in \cite{LaRaZu}, one then deduces that, provided that $\Mink A<Q$, $\ze_A$ and $\zet_A$ have the same poles in $G$ (with the same multiplicities)---and hence, define the same set of visible complex dimensions in $G$; i.e., $\mathscr{P}(A)=\mathscr{P}(\ze_A,G)=\mathscr{P}(\zet_A,G)$. Also, if, in addition, $D\in G$, then we have, in particular, \[\text{res}(\ze_A(\cdot,A_\delta),D)=(Q-D)\text{res}(\zet_A(\cdot,A_\delta),D) .\] Furthermore, it follows from the theorem about integrals depending holomorphically on a parameter that $\mathscr{P}(A)$ is independent of the choice of $\delta>0$ in the definition of either $\ze_A=\ze_A(\cdot,A_\delta)$ or of $\zet_A=\zet_A(\cdot,A_\delta)$ and that, in particular, the above residues of $\ze_A$ and of $\zet_A$ at $s\in D$ are independent of the choice of $\delta>0$. This latter fact will be used in the statement and the proof of Theorems \ref{res} and \ref{th: 12} below.

In the following theorems (analogs of \cite[Theorem 2.2.3]{LaRaZu} and \cite[Theorem 2.2.14]{LaRaZu}, resp.), we establish the relationship between the residues of the distance and tube zeta functions with the Minkowski content of the set $A$. We use the notation $\ze_A(\cdot,A_\de)$ to emphasize the dependence of the value of $\delta$. In particular, we prove $\delta$ has no significance on the value of the residues, just as it does not impact the complex dimensions.

\begin{theorem}\label{res}
	Assume that the bounded subset $A$ of $E$ is Minkowski nondegenerate, i.e., $0<\M_*^D(A)\leq\M^{*D}(A)<\infty$ for $D:=\text{dim}_B A$ and  $D<Q$. If $\ze_A(s)=\ze_A(s,A_\de)$ can be extended meromorphically to a neighborhood of $s=D$, then $D$ is necessarily a simple pole of $\ze(s,A_\de)$, and 
	\[ (Q-D)\M_*^D(A)\leq\text{res}(\ze_A,D)=\text{res}(\ze_A(\cdot,A_\de),D)\leq(Q-D)\M^{*D}(A). \]
	Furthermore, the value of $\text{res}(\ze_A(\cdot,A_\de),D)$ does not depend on $\de>0.$ In particular, if $A$ is Minkowski measurable (i.e., if $\M_*^D(A)=\M^{*D}(A)<\infty$), then
	\[ \text{res}(\ze_A(\cdot,A_\de),D)=(Q-D)\M^D(A).\]
\end{theorem}

\begin{proof}
	Since $\M_*^D(A)>0$, using Theorem \ref{th: 5}(c) we conclude that $s=D$ is a pole of $\zeta_A$. Therefore, it suffices to show that the order of the pole at $s=D$ is not larger than 1. Let us take any fixed $\de>0$, and let
	\[ C_\de := \sup_{t\in(0,\de]}\dfrac{|A_t|}{t^{Q-D}}.\]
	Note that $C_\de<\infty$ because $M^{*D}(A)<\infty$. Then, for $s\in\R$ with $D<s<Q$, we have
	\begin{equation}\label{eq: 2}
		\begin{split}
			\ze_A(s,A_\de)&=\int_{A_\de} d(x,A)^{s-Q}\text{d}x=\de^{s-Q}|A_\de|+(Q-s)\int_0^\de t^{s-Q-a}|A_t|\text{d}t \\
			&\leq C_\de \de^{s-D}+C_\de (Q-s)\dfrac{\de^{s-D}}{s-D}=C_\de (Q-D)\de^{s-D}\dfrac{1}{s-D}.
		\end{split}
	\end{equation}
	Therefore, $0<\ze_A (s,A_\de)\leq C_1(s-D)^{-1}$ for all $s\in(D,Q).$ This shows that $s=D$ is a pole of $\ze_A(s,A_\de)$ which is at most of order 1, and the first claim is established. Namely, $D$ is a simple pole of $\zeta_A$.
	
	As was alluded to earlier, by using the theorem about integrals depending holomorphically on a complex parameter, it is easy to see that for any positive real numbers $\de$ and $\de_1$, with $\de<\de_1$, the difference
	\[ \ze_A(s,A_{\de_1})-\ze_A(s,A_\de)=\int_{A_{\de_1}\setminus A_\de} d(x,A)^{s-Q}\text{d}x \]
	is an entire function of $s$, since $\de\leq d(x,A)\leq\de_1$ for any $x\in A_{\de_1}\setminus A_\de .$ (See also Theorem \ref{th: 8} above.) Therefore, the residue of $\ze_A(s,A_\de)$ at $D$ does not depend on $\de$.
	
	In order to prove the second inequality, it suffices to multiply (\ref{eq: 2}) by $s-D$, with $s$ real, and take the limit as $s\rightarrow D^+$:
	\[ \text{res}(\ze_A(\cdot,A_\de),D)\leq (Q-D)\lim_{s\rightarrow D^+} C_\de \de^{s-D}=(Q-D) C_\de. \]
	Since the residue at $D$ does not depend on $\de$, we establish the second inequality by recalling the definition of $C_\de$ and passing to the limit as $\de\rightarrow 0^+$. The first inequality is proved analogously.
\end{proof}

The following is the counterpart of the above result (Theorem \ref{res}) stated for the tube zeta functions, and is a direct consequence of Lemma \ref{th: 2} and the above proof. 

\begin{theorem}\label{th: 12}
	Assume that $A$ is a Minkowski nondegenerate bounded subset of $\E$ such that $D:=\text{dim}_B A<Q$, and there exists a meromorphic extension of $\zet_A=\zet_A(\cdot,A_\delta)$. Then $D$ is a simple pole of $\zet_A$, and for any positive $\de$, $\textnormal{res}(\zet_A,D)$ is independent of $\de$. Furthermore, we have
	\[ \M_*^D(A)\leq\textnormal{res}(\zet_A,D)\leq \M^{*D}(A). \]
	In particular, if $A$ is Minkowski measurable, then the residue of the tube zeta function of $A$ at $s=D$ is equal to the $D$-dimensional Minkowski content of $A$; that is,
	\[ \textnormal{res}(\zet_A,D)=\M^D(A).\]
\end{theorem}

Note that more effort is required to study the case where $D=Q$, since such a value can be a pole of the tube zeta function but could be missed by the distance zeta function; see the above function equation \ref{functeqn} when $s=Q$. See \cite{LaRaZu} for discussion of the Euclidean case, for which the Ahlfors proof is quite similar.

We end this section by defining (much as in \cite[Chapter 4]{LaRaZu}) the {\em relative zeta functions} for Ahlfors spaces, along with the main theorem adapted to this setting.

\begin{definition}
	Let $\om $ be an closed subset of $\E$, not necessarily bounded, but of finite ($Q$-dimensional) Hausdorff measure. Let $A\subset\E$, also possibly unbounded, such that $\om$ is contained in $A_\de$ for some $\de>0$. The {\it distance zeta function $\ze_A (\cdot,\om)$ of $A$ relative to $\om$} (or {\it the relative distance zeta function}) is defined by
	\[ \ze_A(s,\om)=\int_\om d(x,A)^{s-Q}\text{d}x ,\]
	for all $s\in\C$ with $\textnormal{Re }s$ sufficiently large. We will call the ordered pair $(A,\om)$ a {\it relative fractal drum}.
	
	We can, of course, define similarly the {\em tube zeta function of $A$ relative to $\Omega$} (or the {\em relative tube zeta function}), $\zet_A(\cdot,\Omega)$.
\end{definition}

Note that, in the case of the relative fractal drum $(A,A_\delta)$, we recover the standard setting of the distance zeta function (and similarly for the relative tube zeta function $\zet_A(s,\Omega))$, which is defined entirely analogously) and which generalizes $\ze_A$ in the special case of $(A,\Omega)=(A,A_\delta)$. In this manner, relative fractal drums and their relative zeta functions are direct generalizations of $\ze_A$ and of $\zet_A$.. However, relative fractal drums allow for a study of specific subsets of the tubular neighborhood. Relative fractal drums provide several interesting examples in Euclidean spaces, such as spaces that have negative values or even $-\infty$ as their relative Minkowski dimension. They can also be used to suitably decompose a given bounded set $A$ and compute the fractal zeta functions and the complex dimensions of $A$. See \cite[Chapters 4 and 5]{LaRaZu} and \cite{LaRaZu4} for a detailed exploration of the Euclidean case.

In this way, we can define the relative Minkowski content and dimension:

\begin{definition}
	The {\it upper r-dimensional Minkowski content of $A$ relative to $\om$} is defined as 
	\[\mathcal{M}^{*r}(A,\om)=\limsup_{t\rightarrow 0}\dfrac{|A_t \cap \om|}{t^{Q-r}}.\]
	We then define the {\it (relaltive) upper Minkowski dimension} by
	\begin{align*}
		\Mink (A,\om)&=\inf\{r\in\R:\mathcal{M}^{*r}(A,\om)=0\}=\inf\{r\in\R:\mathcal{M}^{*r}(A,\om)<\infty\}
		\\ &=\sup\{r\in\R:\mathcal{M}^{*r}(A,\om)=+\infty\}.
	\end{align*}
\end{definition}

We can thus define the relative distance zeta function and prove that its main properties still hold:

\begin{theorem} \label{th: 6}
	Let $\om$ be an open subset of $\E$ of finite Hausdorff measure, and let $A\subset\E$ be such that $\om\subset A_\de$ for some $\de>0$. Then:
	
	$(a)$ The relative distance zeta function $\ze_A(s,\om)$ is holomorphic in the half-plane $\{\textnormal{Re} s>\Mink (A,\om)\}$, and for all complex numbers $s$ in that region,
	\[\ze_A'(s,\om)=\int_\om d(x,A)^{s-Q}\log d(x,A)\textnormal{d}x .\]
	
	$(b)$ The lower bound of absolute convergence is optimal, i.e.
	\[ \Mink (A,\om)=D(\ze_A(\cdot,\om)),\]
	where $D(\ze_A(\cdot,\om))$ is the abscissa of absolute convergence of $\ze_A(s,\om)$.
	
	$(c)$ If the Minkowski dimension $D:=\text{dim}_B (A,\om)$ exists, $D<Q$, and $\M_*^D(A,\om)>0$, then $\ze_A(s,\om)\rightarrow+\infty$ as $s\in\R$ converges to $D$ from the right. In this case, the abscissa of holomorphic continuation coincides with the abscissa of absolute convergence: $D_\text{hol}(\ze_A(\cdot,\Omega))=D(\ze_A(\cdot,\Omega))=\text{dim}_B(A,\Omega).$
\end{theorem}

\begin{proof}
	The proof is similar to that of Theorem $\ref{th: 5}$, as the Harvey-Polking Lemma generalizes to the case of relative fractal drums (\cite{Zu4,LaRaZu}).
\end{proof}

If the difference in two domains $\Omega_1$ and $\Omega_2$ is ``nice enough'' (such as both sets including the entirety of $A_\epsilon$ for $\epsilon$ small), then we define the relative fractal drums to be equivalent. Here again, the proof is similar to the Euclidean case and is thus omitted.

\begin{prop}
	Assume that $(A,\om_1)$ and $(A,\om_2)$ are relative fractal drums in $\E$ such that $f_j(s):=\int_{\om_j\setminus(\om_1\cap\om_2)} d(x,A)^{s-Q}\text{d}x$ are entire functions, for $j=1,2$. Then the corresponding distance zeta functions are equivalent,
	\[ \ze_A(\cdot,\om_1)\sim\ze_A(\cdot,\om_2).\]
\end{prop}

Finally, we define the complex dimensions of a relative fractal drum, in an analogous manner to the standard distance zeta function.

\begin{definition}
	Assume that $(A,\om)$ is a relative fractal drum in $\E$ such that its distance zeta function possesses a meromorphic extension to a domain $G$ which contains the critical line $\{ \text{Re }s=D(\ze_A(\cdot,\om))\}$ in its interior. The set of poles of $\zeom$ located on the critical line is called {\it the set of principal complex dimensions of the relative fractal drum $(A,\om)$}, or {\it the set of relative principal complex dimensions of $(A,\om)$}, and is denoted by  $\mathscr{P}_c (\zeom).$ Furthermore, the set of poles of $\ze_A(\cdot,\om)$ belonging to $G$ is called the set of {\em visible complex dimensions of }$(A,\om)$ (relative to $G$) and is denoted $\mathscr{P}(A,\om)$.
\end{definition}

\begin{remark}
	Entirely analogous results and definitions can be obtained for the relative tube zeta functions $\zet_A(\cdot,\Omega)$. In particular, if $\overline{\text{dim}}_B(A,\Omega)<Q$, then, in light of the counterpart of the functional equation {\upshape(\ref{functeqn})} for relative fractal drums, the complex dimensions (and their multiplicities) are the same whether they are defined via $\zeta_A(\cdot,\Omega)$ or via $\zet_A(\cdot,\Omega)$.
\end{remark}


\section{Examples}

Ahlfors spaces occur throughout mathematics in a variety of forms, such as Riemannian and sub-Riemannian manifolds, self-similar fractals in $\R^N$, and harmonic embeddings of fractals. In this section, we will look at two specific examples, the Heisenberg groups and Laakso spaces, both to calculate the complex dimensions of certain subsets as well as to illustrate the difficulties that arise.

\subsection{The Heisenberg Groups}

\begin{definition}
	The {\it n-dimensional Heisenberg group} $\He^n$ in its algebra representation is thought of as $\mathbb{H}^n=\R^n\times\R^n\times\R$ with group multiplication given by 
	\[h\cdot h'=(x,y,t)\cdot(x',y',t')=(x+x',y+y',t+t'+2\sum_{k=1}^n (x_k'y_k-x_ky_k')).\]
	
	There is a natural dilation action $\partial_r(x,y,t)=(rx,ry,r^2t),$ $r>0$, that gives rise to the homogeneous norm called the {\it Heisenberg gauge}, 
	\[\|(x,y,t)\|=\left( \sum_{k=1}^n (x_k^2+y_k^2)^2+t^2\right)^\frac{1}{4}.\]
	This norm has the properties that $\|\partial_r(x,y,t)\|=r\|(x,y,t)\|$ and $\|h^{-1}\|=\|h\|$, where $h^{-1}=(-x,-y,-t)$.
	
	This norm defines a metric on $\He^n$, $d(x,y)=\|x^{-1}y\|$.
\end{definition}

The Haar measure on $\He^n$, denoted by $|\cdot|$ and given by Lebesgue measure (up to a constant multiple), yields
\[|B_r(x)|=r^{2n+2}|B_1(x)|.\]
The Haar measure is invariant under left multiplication.

Thus the Heisenberg group is Ahlfors regular of Ahlfors dimension $Q=2n+2$ (i.e. it is $(2n+2)$-regular), while its topological dimension is $T=2n+1$. More information on these spaces, and the more general Carnot--Carath\'eodory spaces, can be found in \cite{Bj-Bj,DaMcCS,DaGaNh}.

In what follows, we will be working with the space $\He^1=\He$, and denoting points by their ordered triples. Note that $\He$ is Alhfors 4-regular. Also note there will be a few notational differences due to the standard usage of $t$ as the last coordinate in the Heisenberg setting.

\begin{example}
	
	If we fix any point $(x_1,y_1,t_1)=p$ in $\He$, then we know that $|A_r|=|(B_r(p))|=r^4|(B_1(p))|$, where $A=\{p\}$, and thus we can explicitly calculate the complex dimension of the point $p$ by using the tube zeta function $\zet_A=\zet_p$:
	
	\begin{align*}
		\zet_p(s)&=\int_0^r w^{s-4-1} |A_w| \text{d}w= |(B_(p))|\int_0^r w^{s-4-1}w^4 \text{d}w\\
		&= |(B_(p))| \dfrac{r^s}{s}
	\end{align*}
	which has a simple pole at $s=0$. Thus the set of complex dimensions of the point $p$ is $\{0\}$.
	
\end{example}

\begin{example}
	
	Let us take as our set $A\subseteq \He$ the line segment starting at the origin and going to (0,0,1) (or more generally, any line segment that lies on the t-axis). Explicitly, the distance function is given by 
	\begin{equation} \label{eq: dist}
		d((x,y,t),(x_1,y_1,t_1))=\sqrt[4]{((x-x_1)^2+(y-y_1)^2)^2+(t-t_1+2yx_1-2xy_1)^2}.
	\end{equation}
	Since we are along the t-axis, this equation simplifies to the much neater form:
	\[ d((x,y,t),(x_1,y_1,t_1))=((x^2+y^2)^2+(t-t_1)^2)^{1/4}, \]
	and we see that at any given point on our line segment, the horizontal cross-section is a circle. Thus $A_r$ will be a standard cylinder with the hemispheres of the Heisenberg unit ball as the caps.
	
	Again using the tube zeta function, we see that
	\begin{align*}
		\zet_A(s)&=\int_0^r w^{s-4-1}|A_w|\text{d}w 
		= \int_0^r w^{s-4-1} (\pi w^2\cdot1 + w^4|B_1(p)|)\\
		&=\pi \dfrac{r^{s-2}}{s-2}+|B_1(p)|\dfrac{r^s}{s},
	\end{align*}
	which has simple poles at $s=2$ and $s=0$, from which we conclude that our set of complex dimensions is $\mathscr{P_A}=\{0,2\}.$
	
	The pole at $s=0$ can be interpreted as the dimension of the endpoints of the line segment. Interestingly, this tells us that a line segment along the t-axis in the Heisenberg group $\He=\He^1$ effectively is of dimension 2.
	
\end{example}

Line segments elsewhere in the plane are much harder to determine the complex dimensions of, due to the twisting nature of the Heisenberg group. However, we conjecture that a line segment entirely along the x or y axes of $\He$ should be of dimension 1, which would then imply a dimensional shift as line segments become ``more vertical''.

It is also possible that a one-dimensional ``line segment'' in the Heisenberg group is a curve that accounts for this twist, so that the $\delta$-neighborhood around the curve is easily visualized as a curved cylinder.

\subsection{Laakso Spaces}

In 2000, Tomi Laakso published a paper describing a construction of spaces now known as Laakso spaces (\cite{Laa}). These are $Q$-regular path spaces for any $Q>1$ that also satisfy a weak (1,1)-Poincar\'e inequality; see e.g. \cite{Hei} and \cite{HeKo} for an exploration of this condition. Satisfying such an inequality means that Laakso spaces enjoy many Euclidean properties while providing a class of Ahlfors regular spaces for {\em any} dimension $>1$. See \cite{Stein} for further study of analytic properties arising in this space.
We construct here a simple Laakso space based off of the $1/4$-th Cantor set.

We first construct the $1/4$-th Cantor set, called $K$ henceforth, viewed as an iterated function system (IFS). Our IFS is described by the following two similarity maps of $\R$:
\[ \phi_1(x)=\frac{x}{4},~~\phi_2(x)=\frac{x}{4} + \frac{3}{4}. \]
$K$ is then the unique nonempty compact subset of $\R$ such that $\phi_1(K)\cup\phi_2(K)=K$, with all the usual properties of the Cantor set with a scaling factor of 1/4 (instead of 1/3 for the classic ternary Cantor set). It is self-similar and is a topological Cantor set, i.e. it is perfect and totally disconnected --- and hence, it is uncountable, with the cardinality of the continuum. Further, $K$ has Hausdorff dimension $Q_K=\dfrac{\log2}{\log4}=\frac{1}{2}$.

The IFS construction is important for us as it gives a unique addressing scheme for any point in $K$. Specifically, any point can be identified to an infinite binary string $(i_j)_{j=1}^\infty$, where $i_j\in\{1,2\}$. This string represents the composition of maps $\phi_1$ and $\phi_2$ such that our given point is the limit of compositions, $\lim_{n\rightarrow\infty} \phi_{i_n}\circ\dots\circ\phi_{i_2}\circ\phi_{i_1}([0,1]).$ We use the notation $\circ_{j=1}^\infty \phi_{i_j} ([0,1]) $ to represent this limit of compositions.

Using finite binary strings, we can define subsets of $K$ in terms of the different level approximations (or prefractal approximations) under the IFS. We define the subset $K_a$, where $a$ is a finite string $(i_j)_{j=1}^n$, as the set of points in $\circ_{j=1}^n \phi_{i_j} (K) = \phi_{i_n}\circ\dots\circ\phi_{i_2}\circ\phi_{i_1}(K)$. For example, letting $i_1=2$ and $i_2=1$, then $K_{12}$ is the subset formed by first applying $\phi_2$ followed by $\phi_1$: \[K_{12}=\phi_1(\phi_2(K))=\frac{1}{4}\bigg(\frac{1}{4}K+\frac{3}{4}\bigg).\]

\begin{figure}[]
	\centering
	
	\includegraphics[width=0.6\textwidth]{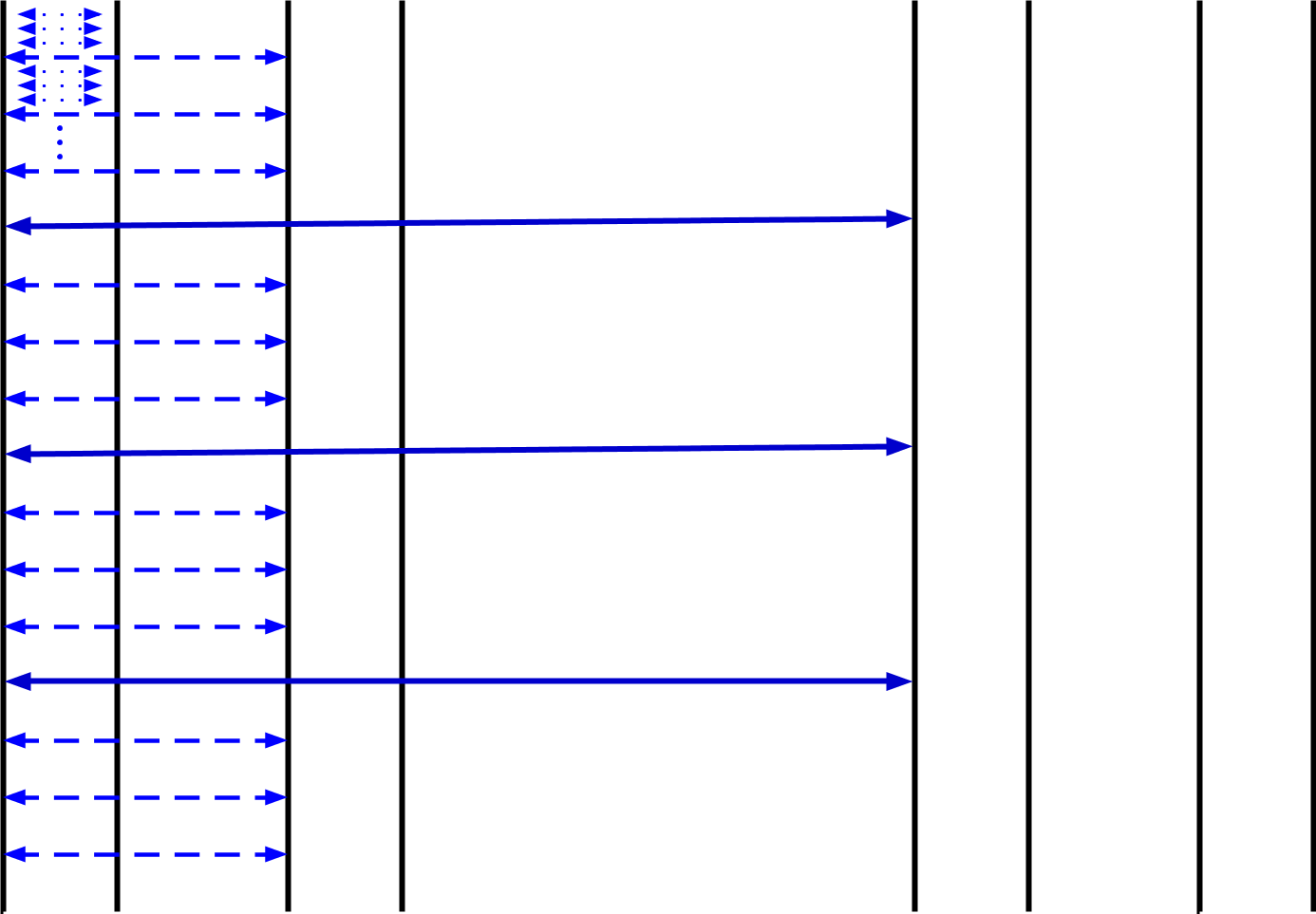}
	
	\caption{A representation of the third-level iterate of our Laakso space, with the left-hand wormholes identified.}
	\label{fig: 1}
\end{figure}
The base space for our Laakso space will then be $K\times I$ in $\R^2$, where $I$ is the unit interval $[0,1]$, with the metric induced from $\R^2$. For our eventual quotient space, we restrict ourselves to a fairly simple choice of identification pairs, otherwise known as {\it wormholes}. Our construction will place wormholes at heights of 1/4-th powers, where 1/4 is equal to the scaling ratio that defines the Cantor set $K$, see Figure \ref{fig: 1}.

Starting with the first identifications, we define the first level wormhole function
\[ \omega(i)=4^{-1}i,~~~0\leq i < 4, \]
so that $\omega(0)=0$, $\omega(1)=1/4$, and so on. We continue this construction for the second level function defined as
\[ \omega(m_1,m_2)=4^{-1}(m_1+4^{-1}m_2),~~~0\leq m_i < 4, \]
giving 16 points from 0 to 15/16. In general,
\[ \omega(m_1,\ldots, m_k) = \sum_{i=1}^k m_i \prod_{h=1}^i 4^{-1} .\]

Provided $m_k > 0$ to avoid overlap, we call the sets $K \times \omega(m_1,\ldots,m_k)$ in $K\times I$ the {\it wormhole levels of order} $k$. Given any finite binary string of length $k-1$, called $a$, identify $K_{a1}\times \omega(m_1,\ldots,m_k)$ with $K_{a2}\times \omega(m_1,\ldots,m_k)$. Taking all such identifications, the resulting quotient space is our Laakso space $F$; see Figure \ref{fig: 1} for a representation of the wormholes up to order 3 along the interval $\{0\}\times I$. Denote the natural projection from $K\times I\rightarrow F$ by the map $s$.

The Laakso space can also be defined as a limit of graphs, with identifications made at each step (see \cite{Ba-Ev}).

The Ahlfors measure on $F$ will essentially be the same as the probability measure on $K\times I$ (i.e. the push forward of the normalized Hausdorff measure $H^{1/2}$ on $K$ crossed with the Lebesgue measure on $I$), since only pairs of points are identified at a time. Define a metric on $F$ as 
\[ d(x,y):=\inf\{H^1(p)| s(p) \text{is a continuous path joining $x$ and $y$}\}, \]
where $p \subset K\times I$ is the pre-image of a path composed of countably many vertical line segments that are connected via wormholes in the quotient space $F$ and $H^1(p)$ is the one-dimensional Hausdorff measure (or simply, the sum of the lengths) of the path $p$. Note that $H^1$ (used on the pre-image) is different from the Ahlfors measure used on $F$; we use the above definition to match historical literature on Laakso space. As the metric is a sum of vertical distances, since horizontal jumps along wormholes carry no distance, we define the height of $x\in F$ as $h(x)$. Thus, if $x,y$ can be connected by an only upward or only downward path, then
\[ d(x,y)=|h(x)-h(y)|. \]
If $p$ is an s-image of a countable number of line segments, then the length of $p$ will simply be the sum of the lengths of the line segments. The following theorem importantly guarantees and classifies the existence of geodesics in $F$ (see \cite[Proposition 3 and Theorem 1]{Cap}). 

\begin{theorem}
	Let $x,y\in F$. There exists a path connecting $x$ and $y$ that requires no more than two turning points. The minimal such path is a geodesic pf $F$.
\end{theorem}

\begin{figure}[t]
	
	\centering       

	\begin{tikzpicture}[x=0.75pt,y=0.75pt,yscale=-1,xscale=1]

	\draw    (30,00) -- (30,120) ;
	 
	\draw    (170,00) -- (170,120) ;
	 
	\draw    (200,00) -- (200,120) ;
	 
	\draw    (60,00) -- (60,120) ;
	
	\draw    (40,00) -- (40,120) ;
 
	\draw    (70,00) -- (70,120) ;

	\draw    (180,00) -- (180,120) ;
	
	\draw    (210,0) -- (210,120) ;

	\draw[thick, color=blue, ->] (28,50) -- (28,60);
	\draw [dashed, color=blue, ->]   (30,60) -- (170,60) ;
	
	\draw[thick, color=blue, ->] (168,60) -- (168,50);
	\draw [dashed, color=blue, ->]   (170,50) -- (200,50) ;

	\draw[thick, color=blue,->] (198,50) -- (198,58.5);
	\draw [dashed, color=blue, ->]   (200,58.5) --(210,58.5) ;
	\end{tikzpicture}

	\caption{An example of a geodesic connecting the point $(0,\frac{129}{256})$ to $(\frac{63}{64},\frac{33}{64})$ by passing through the wormholes at $(0,\frac{1}{2})$ (arrow exaggerated), $(\frac{3}{4},\frac{3}{16})$ and $(\frac{15}{16}, \frac{33}{64})$, using a third level Laakso space approximation.}
	\label{fig: path}
	
\end{figure}

With this metric and the stated Ahlfors measure, $F$ is thus an Ahlfors space of dimension $Q=3/2$. With the space defined, we now look at several subsets of $F$ and compute their complex dimensions.

\begin{example}\label{lak}
	Let us first look at a single point. Let $A = \{(0,1/4)\}$ and let $\delta = 1/4$ (for an explicit calculuation). Using the tube zeta function, we have
	\[ \zet_A (s)= \int_0^\delta t^{s-1/2}|B_t(A)|\text{d}t ,\]
	where $B_t(A)$ is the ball of radius $t$ centered at point $A$, and $|\cdot|$ is the Ahlfors measure on $F$.
	
	Note that since our point is a wormhole of the largest level, any geodesic connecting to any other point in the space will have at most {\emph one} turning point.

	We find that $|B_r(A)|=2r$ for $2/16\leq r \leq 1/4$, as these neighborhoods include all of $ K\times[1/4+r,1/4-r]$ and we are using a probability measure with total measure of 1. Once $1/16\leq r<2/16$, the neighborhood can no longer reach $(3/16, 1/4)$ or $(15/16,1/4)$, but, for any integer $k>2$, the neighborhood {\em does} cross through at least one point of $4^{-k}$, and so we can view the segments along $ K|_{[3/16,1/4]}\times[0,1/2]$ and $K|_{[15/16,1]}\times[0,1/2] $ as "splitting in two"; see Figure \ref{fig: 2}. Note that each region of $K$ in this scenario has a weighted measure of $1/4$, as each region represents a fourth of the full Cantor set. For $1/16\leq r\leq 2/16$, we thus see that $|B_r(A)|=\frac{1}{4}2(2r)+\frac{1}{4}(4(r-\frac{1}{16}))$, where we use a weight of $1/4$ for each region, with the first term representing two segment regions of length $2r$ and the second term in the addition represents 4 segment regions of length $(r-1/16)$, as the paths must travel up or down from (0,1/4) (or (3/4, 1/4)) a distance of 1/16 to reach the first wormhole connecting to any point in these regions. 
	
	\begin{figure}[t]
		
		\centering
		
		\begin{tikzpicture} [scale = 10]

			\draw [very thick] (0,1/6) -- (0,2/6);
			\draw [dashed] (0,1/6) -- (1/16,1/6);
			\draw [dashed] (0,2/6) -- (1/16,2/6);
			\draw [dashed] (1/16,1/6) -- (1/16,2/6);	
			\draw (-.01,1/4) -- (.01,1/4);
			\node [left] at (0,1/4){(0,$\frac{1}{4}$)};
			\draw [very thick] (3/4,1/6) -- (3/4,2/6);
			\draw [dashed] (3/4,1/6) -- (13/16,1/6);
			\draw [dashed] (3/4,2/6) -- (13/16,2/6);
			\draw [dashed] (13/16,1/6) -- (13/16,2/6);
			\draw [very thick] (3/16,28/96) -- (3/16,32/96);
			\draw [dashed] (3/16,28/96) -- (4/16,28/96);
			\draw [dashed] (3/16,32/96) -- (4/16,32/96);
			\draw [dashed] (4/16,28/96) -- (4/16,32/96);
			\draw [very thick] (3/16,16/96) -- (3/16,20/96);
			\draw [dashed] (3/16,16/96) -- (4/16,16/96);
			\draw [dashed] (3/16,20/96) -- (4/16,20/96);
			\draw [dashed] (4/16,16/96) -- (4/16,20/96);
			\draw [very thick] (15/16,28/96) -- (15/16,32/96);
			\draw [dashed] (15/16,28/96) -- (16/16,28/96);
			\draw [dashed] (15/16,32/96) -- (16/16,32/96);
			\draw [dashed] (16/16,28/96) -- (16/16,32/96);
			\draw [very thick] (15/16,16/96) -- (15/16,20/96);
			\draw [dashed] (15/16,16/96) -- (16/16,16/96);
			\draw [dashed] (15/16,20/96) -- (16/16,20/96);
			\draw [dashed] (16/16,16/96) -- (16/16,20/96);
			
		\end{tikzpicture}

		\caption{A representation of the $t=\frac{1}{12}$ neighborhood around the point $(0,\frac{1}{4})$. Note that countably many line segment wormholes are identified with these four segments.}
		\label{fig: 2}
		
	\end{figure}
	
	Similarly, for $2/64 \leq r \leq 1/16$, the neighborhoods contain all of $ K|_{[0,1/16]}\times[1/16,3/16]$ and $ K|_{[3/4,13/16]}\times[1/16,3/16]$ while no longer containing $ K|_{[3/16,1/4]}\times[1/16,3/16]$ and $K|_{[15/16,1]}\times[1/16,3/16] $ at all, and so $|B_r(A)|=\frac{1}{2}2r$. For $1/64 \leq r \leq 2/64$, $|B_r(A)|= \frac{1}{2}(\frac{1}{2}2r)+\frac{1}{2}(\frac{1}{2}2(r-\frac{1}{64})),$ where we rearranged terms to powers of 2, with a similar pattern continuing as $r$ goes to $0$.

	Thus we can solve for $\zet_A$ as follows:
	\begin{align*}
		\int_0^\delta t^{s-1/2} |B_t(A)|\text{d}t &=\sum_{k=1}^\infty \int_{2/4^{k+1}}^{1/4^k}t^{s-1/2}2^{-k+1}2t\text{d}t + \sum_{k=1}^\infty \int_{1/4^{k+1}}^{2/4^{k+1}} t^{s-1/2}2^{-k}(6t-4^{-k})\text{d}t \\
		&= \sum_{k=1}^\infty \dfrac{2^{-k+2}}{s-1/2}\bigg[\bigg(\dfrac{1}{4^k}\bigg)^{s-1/2}-\bigg(\dfrac{2}{4^{k+1}}\bigg)^{s-1/2}\bigg]\\
		&+  2^{-k}\dfrac{6}{s-1/2}\bigg[\bigg(\dfrac{2}{4^{k+1}}\bigg)^{s-1/2}-\bigg(\dfrac{1}{4^{k+1}}\bigg)^{s-1/2}\bigg] \\
		&-  2^{-3k}\dfrac{1}{s-3/2}\bigg[\bigg(\dfrac{2}{4^{k+1}}\bigg)^{s-3/2}-\bigg(\dfrac{1}{4^{k+1}}\bigg)^{s-3/2}\bigg],
	\end{align*}
	where in the second step we merged the series into one as they were absolutely convergent. After some expansion and factoring, we can rewrite this entire series as:	
	\begin{equation}\label{lakseries}
		\left(\dfrac{4(1-2^{-s+1/2})}{s-1/2}+\dfrac{6(2^{-s+1/2})(1-2^{-s+1/2})}{s-1/2}-\dfrac{2^{-s+3/2}(1-2^{-s+3/2})}{s-3/2}\right)\sum_{k=1}^\infty (2^{-2s})^k .
	\end{equation}
	There is a removable singularity that occurs as $s=1/2$ in the first two fractions, and a removable singularity at $s=3/2$ in the third. The geometric series we meromorphically extend to $\C$, using its formula of convergence:
	\[\sum_{k=1}^\infty (2^{-2s})^k = \dfrac{2^{-2s}}{1-2^{-2s}}. \] 
	
	We conclude that $\zet_A$ has a unique meromorphic continuation to all of $\C$ given by $\zet_A(s)=\phi(s)\left(\dfrac{2^{-2s}}{1-2^{-2s}}\right)$, for all $s\in\C$, where $\phi$ is an entire function such that $\phi(s)\neq 0$ whenever $2^{-2s}=1.$ The latter fact can be easily established by considering the explicit expression of $\phi(s)$, given by the term between parentheses in the left-hand side of (\ref{lakseries}). We omit the elementary, but tedious, calculations.
	
	Therefore, the poles of $\zet_A$ thus occur precisely when $1-2^{-2s}=0$, which gives the following set of complex dimensions
	\[ \D = \mathscr{P}_A = \{ 0 + \frac{i \pi  n}{\log 2} : n\in\Z \} = \frac{i \pi \mathbb{Z}}{\log 2}, \]
	with all of the complex dimensions being simple and exact, i.e. actual poles of $\zet_A$.
	
\end{example}

This tells us that wormhole points, while zero dimensional, have the same complex dimensions as those of the 1/4th Cantor set. These complex dimensions capture the geometric oscillations occurring in the neighborhood around the point, which can be pictured as ``branching paths'' across the different wormholes. 

Since the set of wormholes is dense in our space, we expect any point to have the same complex dimensions. It would be interesting to further investigate this question.

\begin{example}
	We now let $A=\{0\}\times K$; that is, the 1/4th-Cantor set that lies on the vertical interval $[0,1]$, and $\delta = 1/4$. Note that for $\de$ between $1/4^k$ and $1/4^{k+1}$, $k\geq 1$, we can view the neighborhood around $A$ as if we were at the $k$-th step of construction of the 1/4-Cantor set. For example, if $k=1$ and $\de$ in the above range, we are essentially viewing the neighborhood for the set $\{ (0,y) : y\in[0,1/4]\cup[3/4,1] \}$, since $\de$ is large enough to make this set and $\{0\}\times K$ indistinguishable in terms of our zeta function. In other words, at the kth level of the 1/4 Cantor set construction, we have $2^k$ ``full intervals'' of length $4^{-k}$, plus $2^k-1$ replicas of Example \ref{lak}, since every two endpoints of the $k$th level approximation together give one side of a point neighborhood.	The tube zeta function is thus	
	\begin{align*} \tilde{\zeta_A}(s)&=\int_0^{1/4} t^{s-3/2-1}|A_t|\text{d}t \\
		&= \sum_{k=1}^\infty \int_{2/4^{k+1}}^{1/4^k} t^{s-3/2-1}2^{-k+1}\left((2^k-1)2t+2^k(4)^{-k}\right)~ \text{d}t 
		\\&+ \sum_{k=1}^\infty  \int_{1/4^{k+1}}^{2/4^{k+1}} t^{s-3/2-1}2^{-k+1}\left((2^k-1)(2t+4t-4^{-k})+ 2^k(4)^{-k}\right)\text~{d}t \\
		&= 2\sum_{k=1}^\infty \left\{\dfrac{2}{s-1/2}\bigg[\bigg(\dfrac{1}{4^k}\bigg)^{s-1/2}-\bigg(\dfrac{2}{4^{k+1}}\bigg)^{s-1/2}\bigg]\right.\\
		&+   \dfrac{6}{s-1/2}\bigg[\bigg(\dfrac{2}{4^{k+1}}\bigg)^{s-1/2}-\bigg(\dfrac{1}{4^{k+1}}\bigg)^{s-1/2}\bigg] \\
		& \left.-\dfrac{4^{k}}{s-3/2}\bigg[\bigg(\dfrac{2}{4^{k+1}}\bigg)^{s-3/2}-\bigg(\dfrac{1}{4^{k+1}}\bigg)^{s-3/2}\bigg]\right\}\\
		&+\sum_{k=1}^\infty \dfrac{2^{-2k+1}}{s-3/2}\left[\left(\dfrac{1}{4^k}\right)^{s-3/2}-\left(\dfrac{1}{4^{k+1}}\right)^{s-3/2}\right]\\
		&-\int_{(0,1/4)_\de} t^{s-3/2-1}|(0,1/4)|\text{d}t.
	\end{align*}
	The first series corresponds to Example \ref{lak}, multiplied by $2^k$. Thus we can write this first series as
	\begin{align*}
		2\left(\dfrac{4(1-2^{-s+1/2})}{s-1/2}+\dfrac{6(2^{-s+1/2})(1-2^{-s+1/2})}{s-1/2}-\dfrac{2^{-s+3/2}(1-2^{-s+3/2})}{s-3/2}\right)\sum_{k=1}^\infty (2^{-2s+1})^k\\
		= \phi(s)\sum_{k=1}^\infty (2^{-2s+1})^k ,
	\end{align*} 
	with the notation for $\phi$ introduced in Example \ref{lak}.
	As before, we have removable singularities at $s=1/2$ and $s=3/2$, but now our geometric series meromorphically extends as
	\[ \sum_{k=1}^\infty (2^{-2s+1})^k = \dfrac{2^{-2s}}{1-2^{-2s+1}},\]
	which has poles at $s=1/2 + \pi i n/ \log2$, for any integer $n\in\Z$. The second series (separated here only for emphasis) can be written as
	\[ \sum_{k=1}^\infty \dfrac{2^{-2k+1}}{s-3/2}\left[\left(\dfrac{1}{4^k}\right)^{s-3/2}-\left(\dfrac{1}{4^{k+1}}\right)^{s-3/2}\right] = \dfrac{1-2^{2s-3}}{s-3/2}\sum_{k=1}^\infty (2^{-2s+1})^k.\] Again, we have a removable singularity at $s=3/2$, and the same series as above shows itself. Finally, we subtract off exactly Example \ref{lak}, and so the same poles will occur. Thus we see that $\zet_A$ has a unique meromorphic extension to all of $\C$  and that the set of complex dimensions of $A$ is given by
	\[ \mathcal{D} = \mathscr{P}_A = \left\{ 0 + \dfrac{\pi i n}{\log2}, \dfrac{1}{2} + \dfrac{\pi i n}{\log2} : n\in\Z\right\} = \frac{i\pi\mathbb{Z}}{\log 2} \cup \left( \frac{1}{2} + \frac{i\pi\mathbb{Z}}{\log 2}\right),\]
	with all the complex dimensions being simple.
	Thus we recapture the dimensions of the 1/4th Cantor set itself, as well as what appears to be the complex dimensions at each of the wormhole points.
\end{example}

\subsection{Laakso Graph}


The Laakso graph, also known as the Laakso diamond space, is very similar to the Laakso space, and can be viewed as the same construction with some {\em segments} identified, rather than single points. Alternatively, we can view the Laakso graph as being constructed in a similar manner to the 1/4th-Cantor set, where instead of removing the middle halves, two copies of the middle halves are created at each step in the construction (see Figure \ref{fig: 6}, part (b)). 

To construct the Laakso graph rigorously, we follow the same procedure as in \cite{ChKl} (see also \cite{LanPl}), starting with $X_0=[0,1]$. For $i>0$, define $X_i$ by replacing each edge of $X_{i-1}$ by a $4^{-(i-1)}$ scaled copy of $\Gamma$ (see Figure \ref{fig: 6}). Then $\{X_i\}_{i=0}^\infty$ forms an inverse system \[X_0\xleftarrow{\pi_0}\ldots\xleftarrow{\pi_{i-1}} X_i\xleftarrow{\pi_i}\ldots,\] where $\pi_{i-1}:X_i\rightarrow X_{i-1}$ collapses the copies of $\Gamma$ at the $i$-th level.

Then the inverse limit $X_\infty$ is the {\bf Laakso graph}, with metric \[d_\infty (x,x')=\lim_{i\rightarrow\infty} d_{X_i}(\pi_i^\infty(x),\pi_i^\infty(x')),\] where $X_\infty$ is the Gromov-Hausdorff limit of $\{X_i\}$ and $\pi_i^\infty:X_\infty\rightarrow X_i$ is the canonical projection. 

This process can also be formalized as an Iterated Graph System, see \cite{AnErShi}. Similar to Iterated Function Systems, the limiting dimension is the log of the number of segments divided by the log of the scaling constant. For the Laakso graph, $Q=\log(6)/\log(4)$ is its regularity dimension.

\begin{figure}[t]

	\centering
	\begin{subfigure}{0.5\textwidth}
		\centering
		\includegraphics[width=160px]{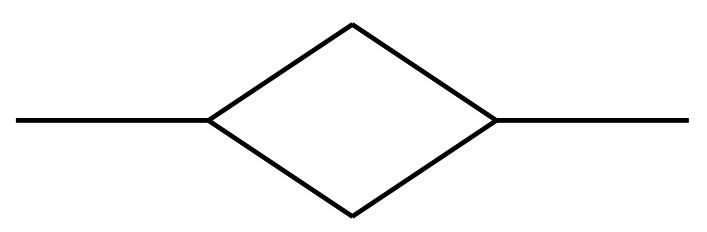}
		\subcaption{ Graph of $\Gamma$ }
	\end{subfigure}%
	\begin{subfigure}{0.5\textwidth}
		\centering
		\includegraphics[width=0.7\textwidth]{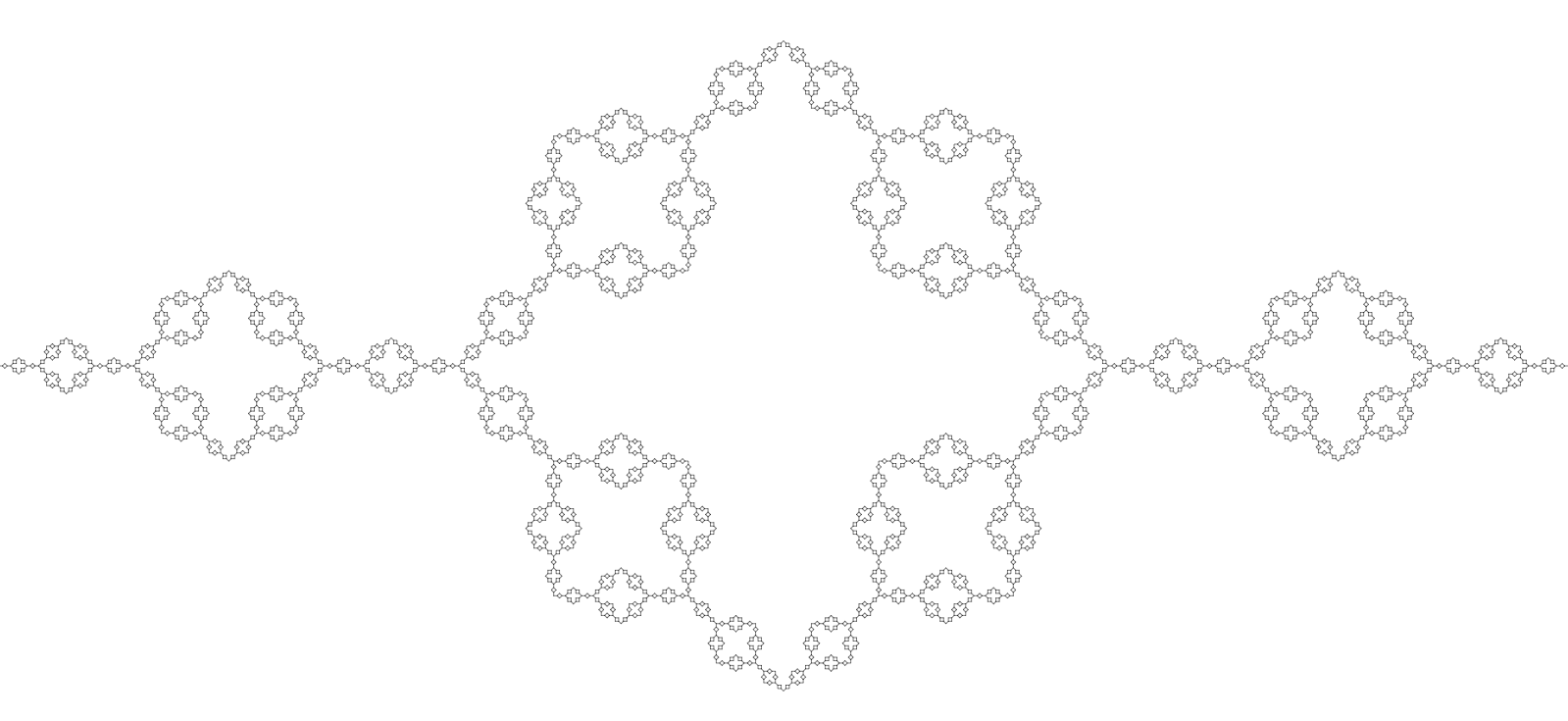}
		\subcaption{Laakso Graph}
	\end{subfigure}
	\caption[Laakso Graph]{}
	\label{fig: 6}
\end{figure}

The Laakso graph is a compact, and hence complete, metric space, which we will make into a metric measure space by using the same Bernoulli probability measure as for the Laakso space. Further, the Laakso graph cannot be  embedded in a bi-Lipschitz manner into any Euclidean space $\R^N$ (with $N\geq1$). While we visualize the graph in $\R^2$, such a visualization requires distortion. We will be working with the graph proper,  with diameter set equal to 1.

Here, we will identify points as an ordered pair using position and a (possibly infinite) word composed of three symbols; $u$ for upper path, $l$ for lower path, and $o$ for the special case of points on the neutral path. So we have $(0,o)$ and $(1/4,o)$, two choices for $1/2$ of $(1/2,u)$ and $(1/2,l)$, and four choices for $3/8$: $(3/8,uu)$, ($3/8,ul)$ and so on.

\begin{example}
	
	Let $A$ be the geodesic connecting the upper most path from $(\frac{1}{4},o)$ to $(\frac{3}{4},o)$, and let $\delta = \frac{1}{4}$. 
	
	The tubular neighborhood contains all branches to the left of $(\frac{1}{4},o)$ and to the right of $(\frac{3}{4},o)$. Since we are using a probability measure, we can visualize all these branches ``collapsing'' into a single segment. We also extend to all paths beneath this upper geodesic. When $t=1/4$, the tubular neighborhood is the entire graph. However, as $t$ decreases, the neighborhood no longer reaches $(\frac{1}{2},l)$. The neighborhood reaches all smaller branches equally, so we can view the neighborhood of all lower paths connecting $(\frac{1}{4},o)$ to $(\frac{3}{4},o)$ as two shrinking line segments, from the point of view of the measure; see Figure \ref{fig: laakso}. Because the measure is a probability measure with total mass equal to 1, the collection of these lower paths has a measure of $\frac{1}{2}\cdot 2t$.
	
	This process continues with each next level of Laakso graph for paths {\em directly connected} to the upper geodesic. When $t<\frac{1}{16}$, the tubular neighborhood no longer reaches $(\frac{3}{4},ul)$ and $(\frac{5}{8},ul)$, and so we can view the lower paths connecting  $(\frac{1}{4},o)$ to $(\frac{1}{2},u)$ and $(\frac{1}{2},u)$ to $(\frac{3}{4},o)$, respectively, as two decreasing line segments for all smaller $t$ values. Similar to above, the measure of these collective paths will be weighted by $\frac{1}{4}$.
	
	\begin{figure}[t]

		\centering
		\begin{subfigure}{0.5\textwidth}
			\centering
			\includegraphics[width=1\textwidth]{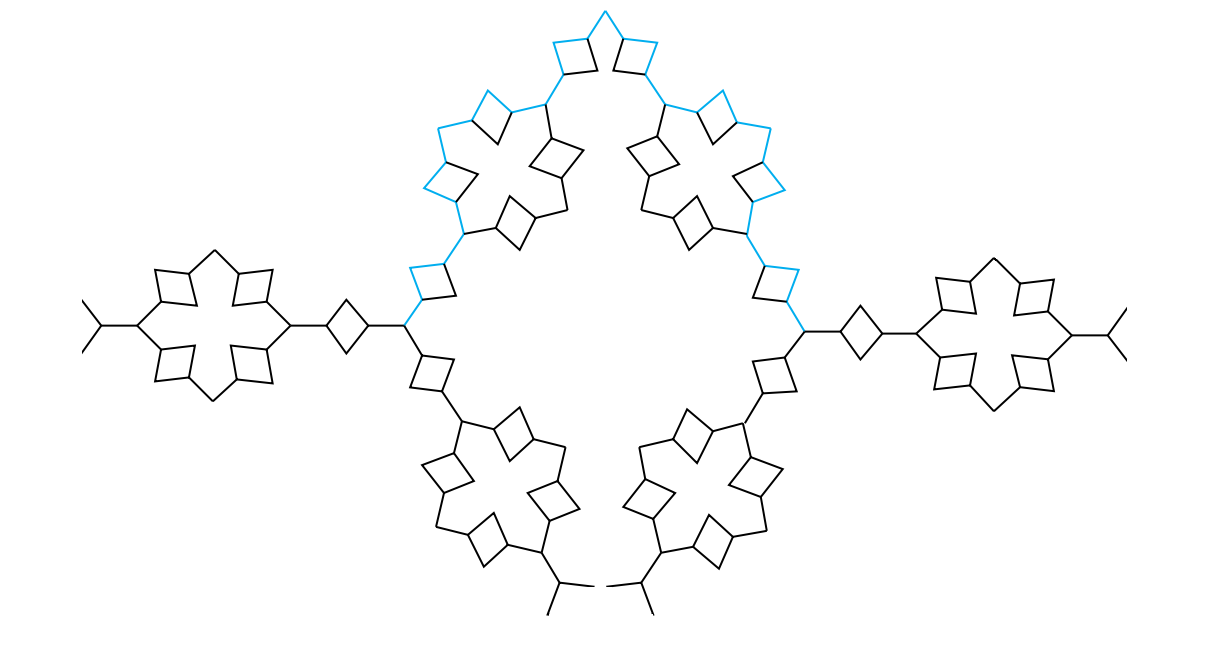}
			\subcaption{ $t=14/64$ neighborhood of $A$ }
		\end{subfigure}%
		\begin{subfigure}{0.5\textwidth}
			\centering
			\includegraphics[width=0.7\textwidth]{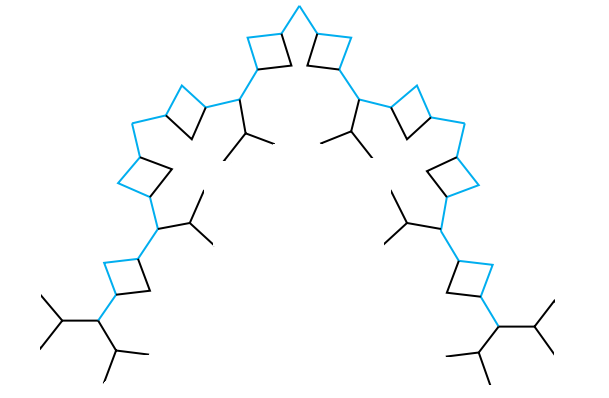}
			\subcaption{$t=2/64$ neighborhood of $A$}
		\end{subfigure}
		\caption[Laakso Graph]{}
		\label{fig: laakso}
	\end{figure}
	
	Since it is easier to keep track of what we {\em remove} as $t$ decreases, we use the above argument by subtraction to compute the measure of the tubular neighborhood. Note that the inner sum is the (measure weight at the $j^{th}$ level) times (number of segments with removals at $j^{th}$ level) times (length removed), and we see that the measure and number of segments directly cancel out. As a reminder, $Q=\log_4(6)$.
	
	\begin{equation}\label{graphseries}
		\begin{split} 
		\tilde{\zeta_A}(s) &= \sum_{k=1}^\infty \int^{4^{-k}}_{4^{-k-1}} t^{s-Q-1} |A_t| \text{d}t\\
		&= \sum_{k=1}^\infty \int^{4^{-k}}_{4^{-k-1}} t^{s-Q-1}\left(2t+\left( \frac{1}{2} - \sum_{j=1}^k 2^{-j}\cdot2^j(4^{-j}-t) \right) \right) \text{d}t\\
		&= \sum_{k=1}^\infty \int^{4^{-k}}_{4^{-k-1}} t^{s-Q-1}\left(2t+ kt + \frac{1}{2} - \left[\frac{1}{3}\left(1-\left(\frac{1}{4}\right)^k\right)\right] \right) \text{d}t\\
		&= \sum_{k=1}^\infty \int^{4^{-k}}_{4^{-k-1}} t^{s-Q-1}\left(2t+ kt + \frac{1}{6} - \frac{4^{-k}}{3} \right) \text{d}t\\
		&= \sum_{k=1}^\infty  \frac{2t^{s-Q+1}}{s-Q+1}+ \frac{kt^{s-Q+1}}{s-Q+1} + \left(\frac{1}{6} - \frac{4^{-k}}{3} \right)\frac{t^{s-Q}}{s-Q} \Big\rvert^{4^{-k}}_{4^{-k-1}} \\
		&= \frac{4^{\log_4{6}-s-1/2}}{s-\log_4{6}+1} + \frac{6}{(4^{s+1}-6)(s-\log_4{6}+1)} + \frac{4^{\log_4{6}-s-1/2}(6+4^{s+1/2})}{3(4^{s+1}-6)(s-\log_4{6})} ,\\
		\end{split}
	\end{equation}
	where we used the meromorphic continuations of the series to simplify. Hence, $\zet_A$ has a unique meromorphic continuation to all of $\C$, given for all $s\in\C$ by the last equality in (\ref{graphseries}).
	
	From these calculations and the fact that $\log_4(6)-1 = \log_4(\frac{3}{2})$, we find that the complex dimensions are simple and given by
	\[ \D =\mathscr{P}_A= \left\{\log_4{\frac{3}{2}}+\dfrac{2\pi i n}{\log 4},  \log_4{6}, \log_4{6}-1 : n\in\Z \right\} ,\] where we note that the dimension $\log_4{6}-1$ appears without nonnreal complex dimensions in the leftmost term of (\ref{graphseries}), which represents the tube zeta calculation on the stable neighborhoods to the left of $(1/4,o)$ and right of $(3/4,o)$. Geometrically, we can view these dimensions as saying this set has dimension $\log_4{6}-1$ from the endpoints, dimension $\log_4{6}$ from being a path with full dimension (which we can see from the calculation of (\ref{graphseries}) has measure 1/6), and oscillating $\log_4{\frac{3}{2}}$ dimensions coming from the branching points along $A$ that occur according to the 1/4-Cantor set.
\end{example}





\section{To Beyond}

\subsection{Patchwork Spaces}

The definition of the distance zeta function does not require the space to necessarily be an Ahlfors space. All that is needed is a metric measure space with an appropriate idea of ``ambient dimension'', or alternatively of the co-dimension, as well as a generalization of Minkowski dimension. In this section we attempt to loosen the restriction on the space in order to understand how the theory should be generalized in order to still give relevant geometric information.

To begin, we define a new type of space, more general than Ahlfors spaces:

\begin{definition}
	A space $X$ is of {\bf Ahlfors-type between dimensions $D_1$ and $D_2$} (here on, {\bf $A(D_1,D_2)$-type}) if it has both a metric $d$ and a measure $\mu$ and there exists $K>0$ such that
	$$K^{-1}r^{D_1}\leq\mu(B_r(x))\leq Kr^{D_2}$$ for all $x\in X, 0<r\leq\text{diam}X$.
\end{definition}

Given a metric measure space (with finite, normalized diameter), we can deform the metric in order to obtain a new metric space. We will use these deformations in order to construct simple examples of Ahlfors-type spaces.

\begin{definition} 
	
	Let $(M,d(x,y))$ be a metric space with diameter $\leq 1$. Given a real number $0<s<1$ consider the space $(M,d(x,y)^s)$. This is a metric space, and the transformation is called the \textit{snowflake functor}.
	\vspace{1em}
	
	Further, if $(M,d(x,y),\mu)$ is regular of dimension $D$, then $(M,d(x,y)^s,\mu)$ is regular of dimension $D/s$; see \cite{Da-Se}.
	
\end{definition}

Using the snowflake functor, we can create new Ahlfors-type spaces by ``patching together'' known spaces with each patch snowflaked in a different manner. For example, the patchwork unit interval, with segment $[0,\frac{1}{4}]$ equipped with the Euclidean metric raised to the 1/2 power, is an $A(1,2)$-type space (see Figure \ref{fig: 3}).

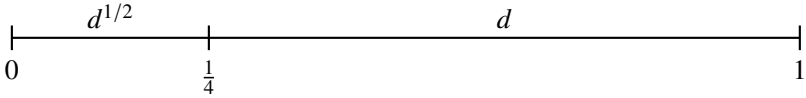
\begin{figure}[t]
	
	\centering
	
	\vfill
	
	\begin{tikzpicture} [scale=.65]
		
		\draw [thick] (0,2) -- (16,2);
		\draw [thick] (4,1.75) -- (4,2.25);
		\draw [thick] (0,1.75) -- (0,2.25);
		\draw [thick] (16,1.75) -- (16,2.25);
		\node [below] at (0,1.75) {0};
		\node [below] at (4,1.75) {$\frac{1}{4}$};
		\node [below] at (16,1.75) {1};
		\node [above] at (2,2) {$d^{1/2}$};
		\node [above] at (10,2) {$d$};
		
	\end{tikzpicture}
	
	\vfill
	
	\caption{A 1-dimensional patchwork space, the line segment $[0,1]$ with the metric on $[0,\frac{1}{4}]$ the standard Euclidean metric raised to the $1/2$ power.}
	\label{fig: 3}
	
\end{figure}

This example of space is of interest in determining the proper generalization of the distance zeta function. In a $Q$-Ahlfors space, $Q$ appears in both the distance and tube zeta functions:
\[ \ze_A(s) = \int_{A_\de} d(x,A)^{s-Q} \text{d}x , \]
\[ \zet_A(s) = \int_0^\delta t^{s-Q-1}|A_t| \text{d}t .\]
How should this exponent be changed if the dimension is no longer constant throughout our space?

We propose that, in order to obtain adequate geometrical data, the ``local dimension'' should be used, and the integral split accordingly. 

\begin{definition}
	Given a space $X$ of  {\bf $A(D_1,D_2)$-type}, we say $x\in X$ is of {\bf local dimension $D \in [D_1,D_2]$} if there exists an $r>0$ such that 
	$$K^{-1}r^{D}\leq\mu(B_r(x))\leq Kr^{D} .$$ If no such $r$ exists, then there exists $D_a < D_b \in [D_1,D_2]$ such that for some $r>0$, $$K^{-1}r^{D_a}\leq\mu(B_r(x))\leq Kr^{D_b} .$$ In such a case, we may use  the infinimum of such $D_a$ as the {\em local dimension} and call $x$ a boundary point.
\end{definition}

With this proposed definition, we can now look at some patchwork space examples:

\begin{example} 
	Take $A = [\frac{1}{8},\frac{1}{2}]$, with $\delta = \frac{1}{4}$, then the tube zeta function (for explicit calculations since $A$ is a full subset of the space) would be realized as:
	
	\begin{align*}
		\zet_A(s) &= \int_0^\delta t^{s-2-1}(t^2+\frac{1}{4}) \text{d}t + \int_0^\delta t^{s-1-1}(t+\frac{1}{4}) \text{d}t\\
		&= \frac{\de^s}{s}+\frac{\de^{s-2}}{4(s-2)}+\frac{\de^s}{s}+\frac{\de^{s-1}}{4(s-1)}.
	\end{align*}
	Hence, $\zet_A$ would have a unique meromorphic continuation to all of $\C$, given for all $s\in\C$ by this last equality.
	
	This would lead to the set of complex dimensions of $A$ being
	\[ \D =\mathscr{P}_A = \{ 0, 1, 2\}, \]
	accounting for the endpoints, the one-dimensional section existing in $[\frac{1}{4}, \frac{1}{2}]$, and the two-dimensional section existing in $[ \frac{1}{8},\frac{1}{4} ]$.
\end{example}

Conversely, the use of a global dimension provides counter-intuitive results. For example, letting $Q=2$ throughout our space would lead any point in $[\frac{1}{4},1]$ to be of dimension 1, and any line segment to be of dimension 2. Not only would the geometric information obtained be muddled, it would not conform to dimensions of line segments and points in standard $\R$. Similar problems occur if $Q=1$ is chosen to be the global dimension, with a line segment in $[0,\frac{1}{4}]$ attaining dimension 0 with endpoints of dimension -1.

\begin{figure}
	\centering
	
	\vfill
	
	\begin{tikzpicture} [scale=.7]
		\draw [thick] (0,0)--(8,0)--(8,8)--(0,8)--(0,0);
		\draw [thick] (4,0)--(4,8);
		\draw [thick] (0,4)--(8,4);

		\draw [fill=gray] (0,0) rectangle (4,4);
		\draw [fill=cyan] (0,4) rectangle (4,5);
		\draw [fill=cyan] (4,4) rectangle (5,0);
		\draw [fill=purple] (4,4) rectangle (6,6);
		
		\node at (2,3) {$d^{1/3}$};
		\node at (2,7) {$d^{1/2}$};
		\node at (6,3) {$d^{1/2}$};
		\node at (6,7) {$d$};
		
	\end{tikzpicture}
	\caption{The $\de = 1/4$ neighborhood of $A$ with the $\ell_\infty$ metric}
	\label{fig: 4}
	
\end{figure}
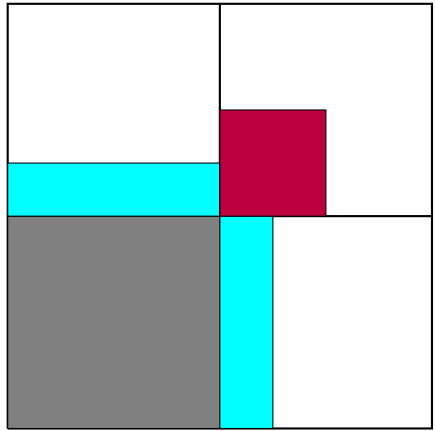

Just as we did with the line segment, we can create new patchwork spaces out of the unit square, or any other subset of $\R^N$ with normalized, finite diameter. Let us take the unit square and split it into fourths, giving each fourth a different snowflake metric (viewed as quadrants, we use $d$ in QI, $d^{1/2}$ in QII and QIV, and $d^{1/3}$ in QIII), with the condition that the shared edge belongs to the minimizing distance. If we use $d$ as the $\ell_\infty$  metric, we will have an $A(1,6)$-type space with the new metric given by the minimizing path .

\begin{example}
	
	Let $A=[0,1/2]\times[0,1/2]$ and $\delta\in(0,1/2)$; see Figure \ref{fig: 4}. Then using the tube zeta function, we find that
	
	\begin{align*}
		\tilde{\zeta_A}(s) &= \int_0^\delta t^{s-6-1}\Big(\frac{1}{4}\Big)\text{d}t + \int_0^\delta t^{s-4-1}\Big(\frac{1}{2}t^2\cdot2\Big)\text{d}t + \int_0^\delta t^{s-2-1}t^2\text{d}t
		\\	&= \frac{1}{4}\dfrac{\delta^{s-6}}{s-6}+\dfrac{\delta^{s-2}}{s-2} + 2 \dfrac{\delta^{s-1}}{s-1} -  \dfrac{t^s}{s}.
	\end{align*}
	Hence, $\zet_A$ would have a unique meromorphic continuation to all of $\C$, given for all $s\in\C$ by this last equality.
	This implies that the set of complex dimensions of $A$ is given by 
	\[ \D =\mathscr{P}_A = \{0,1,2,6\}, \]
	with all of the complex dimensions being simple and exact.
\end{example}

If instead we started with $d$ as the $\ell_1$ (taxi-cab) metric (see Figure \ref{fig: 5}), we find new, initially surprising dimensions, as shown in the next example.

\begin{example}
	
	\begin{align*}
		\tilde{\zeta_A}(s) &= \int_0^\delta t^{s-6-1}\Big(\frac{1}{4}\Big) + t^{s-4-1}(t^2 + (t-t^2)(t-t^2)^2) +  \frac{1}{2} t^{s-2-1}t^2\text{d}t
		\\	&= \frac{1}{4}\dfrac{\delta^{s-6}}{s-6}+\dfrac{\delta^{s-2}}{s-2} +  \dfrac{\delta^{s-1}}{s-1} -  \dfrac{t^s}{s} + 3 \dfrac{t^s+1}{s+1} - \dfrac{t^{s+2}}{s+2}.
	\end{align*}
	Hence, $\zet_A$ would have a unique meromorphic continuation to all of $\C$, given for all $s\in\C$ by this last equality.
	This implies that the set of complex dimensions of $A$ is given by
	\[ \D =\mathscr{P}_A=\{ -2, -1, 0, 1, 2, 6\},\]
	with all of the complex dimensions being simple and exact.
	
\end{example}

\begin{figure}
	\centering
	
	\vfill
	
	\begin{tikzpicture} [scale=.8]
		\draw [thick] (0,0)--(8,0)--(8,8)--(0,8)--(0,0);
		\draw [thick] (4,0)--(4,8);
		\draw [thick] (0,4)--(8,4);
		
		\draw [fill=gray] (0,0) rectangle (4,4);
		\draw [fill=cyan] (0,4) rectangle (4,5);
		\draw [fill=cyan] (4,4) rectangle (5,0);
		\draw [fill=purple] (4,4) -- (6,4) -- (4,6) -- (4,4);
		\draw [fill=red] (4,5) -- (7/2,5) -- (4,6) -- (4,5);
		\draw [fill=red] (5,4) -- (5,7/2) -- (6,4) -- (5,4);
		\node at (2,3) {$d^{1/3}$};
		\node at (2,7) {$d^{1/2}$};
		\node at (6,3) {$d^{1/2}$};
		\node at (6,7) {$d$};
		
	\end{tikzpicture}
	\caption{The $\de = 1/4$ neighborhood of $A$ with the $\ell_1$ metric}
	\label{fig: 5}
	
\end{figure}
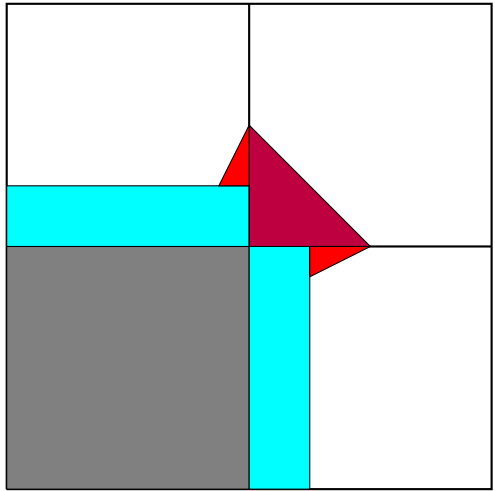

These negative dimensions appear due to geodesics that utilize the lower-dimensional boundary connecting the patches, where in this instance the negative dimensions seem to indicate a ``faster'' travel time than should be possible in standard Ahlfors spaces.

\section{Concluding Comments}

From the above examples, it is clear that the distance (and consequently, tube) zeta function should be extended to more general spaces than simply Ahlfors spaces. Of course, neither the patchwork spaces or the Laakso graph are far removed from Ahlfors spaces, and much is yet unknown about spaces in which the geometry changes throughout the space. While analysis in metric measure spaces is well studied, most examples of such spaces are, in practice, Ahlfors spaces. Of course, the study of such spaces is still recent, and it would be interesting to look at far more general examples than those above.

The above examples show that a notion of ``local dimension'' is necessary if we wish to capture relevant geometric data about any given set. In our examples the local dimension was obvious, but it is still a question of how to define such a dimension in general. It is easy to imagine a space in which the dimension shifts in a continuous manner, but even in this setting it is difficult to both pinpoint the local dimension at any point, and to further show that the complex dimensions of the distance zeta function can be interpreted as geometric information.

It is also currently unknown how general such spaces that admit a local dimension are. It may be possible that, with the correct definition of dimension, the theory of fractal zeta functions can be extended to any doubling metric space, in the sense of the complex dimensions truly giving geometric data. However, it seems likely that, for this purpose, some necessary conditions on the regularity of the space will be needed.

We note that after the work corresponding to this paper was completed, and motivated by the present work and the theory developed in \cite{LaRaZu}, Alexander Henderson has obtained some partial results in this direction \cite{Hen}.

In addition to studying complex dimensions, we believe that the fractal tube formulas, as introduced and established in Chapter 5 of \cite{LaRaZu} (see also \cite{LaRaZu6}), should generalize to Ahlfors spaces in a similar manner. These formulas allow one to reconstruct certain information about a fractal drum by taking a sum over the complex dimensions of the residues of its distance or tube zeta function, viewed as a type of generalization of the classical Steiner tube formulas for compact convex sets. Indeed, examples of tube formulas in Ahlfors spaces may provide the answer to what properties a ``local dimension'' must satisfy in order for further expansion of the theory.

The investigation of such general fractal tube formulas will be carried out in a later work by the authors (\cite{LaWat}).



\begin{funding}
The research of M.L. Lapidus was supported by the Burton Jones Endowed Chair Fund in Pure Mathematics, as well as by grants from the U.S. National Science Foundation (NSF) and by a grant from the French Agence Nationale pour la Recherche (ANR FRACTALS, FANR-24-CE45-3362).
\end{funding}


\end{document}